%% file: main.tex
\documentclass{amsart}
\usepackage[utf8]{inputenc}
\usepackage{amsfonts,color,newclude}
\usepackage{amsthm}
\usepackage{amssymb,transparent}
\usepackage{amsmath}
\usepackage{comment}
\usepackage{graphicx}
\usepackage{todonotes}
\usepackage{verbatim}
\usepackage[colorlinks=true, pdfstartview=FitV, linkcolor=blue, citecolor=blue, urlcolor=blue]{hyperref}

\usepackage{tikz}

\newtheorem{theorem}{Theorem}[section]

\newtheorem{proposition}[theorem]{Proposition}
\newtheorem{cor}[theorem]{Corollary}
\newtheorem{lemma}[theorem]{Lemma}
\newtheorem{example}[theorem]{Example}

\newtheorem{remark}[theorem]{Remark}

\newtheorem{claim}{Claim}[theorem]

\theoremstyle{definition}
\newtheorem{definition}[theorem]{Definition}

\newcommand{\co}{\colon\thinspace}
\newcommand{\Htwo}{\mathbb{H}^2}
\newcommand{\R}{\mathbb{R}}
\newcommand{\reals}{\mathbb{R}}
\newcommand{\PSLR}{\operatorname{PSL}_2(\mathbb{R})}

\newcommand\bu{\mathbf{u}}
\newcommand{\bv}{\mathbf{v}}
\newcommand\bw{\mathbf{w}}
\newcommand{\Z}{\mathbb{Z}}

\definecolor{bettergreen}{rgb}{0,0.6,0.4}
\definecolor{plainblue}{rgb}{.2,.2,1}

\begin{document}

\title[Walks With Jumps]{Walks With Jumps: a neurobiologically motivated  class of paths in the hyperbolic plane}

\author{Jason DeBlois}

\address{Department of Mathematics\\
University of Pittsburgh\\
301 Thackeray Hall\\
Pittsburgh, PA 15260}
\email{jdeblois@pitt.edu}

\author{Eduard Einstein}
\address{Swarthmore College \\ Department of Mathematics and Statistics \\ 500 College Ave. \\ 
Swarthmore, PA 19081}
\email{eeinste1@swarthmore.edu}

\author{Jonathan D. Victor}

\address{Feil Family Brain and Mind Research Institute\\
Weill Cornell Medicine\\
New York, NY 10065}
\email{jdvicto@med.cornell.edu}

\date{\today}

\begin{abstract}
We introduce the notion of a ``walk with jumps'', which we conceive as an evolving process in which a point moves in a space (for us, typically $\Htwo$) over time, in a consistent direction and at a consistent speed except that it is interrupted by a finite set of ``jumps'' in a fixed direction and distance from the walk direction. Our motivation is biological; specifically, to use walks with jumps to encode the activity of a neuron over time (a ``spike train``). Because (in $\Htwo$) the walk is built out of a sequence of transformations that do not commute, the walk's endpoint encodes aspects of the sequence of jump times beyond their total number, but does so incompletely. The main results of the paper use the tools of hyperbolic geometry to give positive and negative answers to the following question: to what extent does the endpoint of a walk with jumps faithfully encode the walk's sequence of jump times?
\end{abstract}

\maketitle

\section{Introduction}

This paper introduces the notion of a \textit{walk with jumps}, which we conceive as an evolving process in which a point moves in a space (for us, typically $\Htwo$) for a fixed duration, in a consistent direction and at a consistent speed except that it is interrupted by a finite set of stereotyped ``jump'' events, each of which moves a fixed distance at a fixed angle from the walk direction. The walk's endpoint encodes partial information about the timing of the jumps, and a main goal here is to offer some precise answers to the question of how it does so.

Our biological motivation, laid out in greater detail in Section \ref{biology}, is to use walks with jumps as a novel abstraction of how a neuron represents information, with the walk itself capturing the neuron's activity over time, and the walk's endpoint considered as an invariant of this activity. Since neural activity consists of a stereotyped series of action potentials that differ only in their timing, the action potentials need to correspond to jumps that are identical in angle and size. 

While the walk with jumps model may somewhat resemble random walk or random flight models, which are well-studied (see eg.~\cite{OrsingherDeGregorioRandomFlights}), we believe that this resemblance is largely superficial and this model's features are substantially different. For instance the fixed jump length and jump angle, choices that follow from the biological motivation, imply that any walk with jumps traces out an \emph{embedded} path in $\mathbb{H}^2$ that is a quasi-geodesic (by Corollary \ref{angle bound} and Proposition \ref{trig qg}, respectively). Indeed, a primary initial motivation for the model's construction was to provide a continuous interpolation for collections of paths that traverse edges of a tree, and embedding trees in $\mathbb{H}^2$ is a natural choice for this purpose. Corollary \ref{the ned} gives an explicit form in which the model successfully accomplishes this goal.

Despite the rigid constraints on the nature of jumps, walks with jumps in $\mathbb{H}^2$ exhibit a rich collection of behaviors as captured by their endpoints: under broad circumstances, the collection of endpoints is the closure of its interior (Corollary \ref{rich behavior}) and has large diameter (Remark \ref{diameter bound}). This reflects fundamental structural differences between the isometry groups of $\mathbb{H}^2$ and $\mathbb{R}^2$. Unlike in the Euclidean case, where the set of translations forms an abelian normal subgroup, hyperbolic translations do not in general commute, and their product is not necessarily a translation. It is due to this that changing only the \emph{timing} of the jumps---not the angle or length---can profoundly affect the walk with jumps's endpoint. 

Another point of departure from existing mathematical literature is that, in keeping with our biological motivation, we focus on finite-time behavior of walks with jumps (hence a finite number of jumps), rather than their asymptotic behavior, and we make no assumptions about the statistics of the jumps (e..g, whether the intervals are Poisson).  In an interval over which information is accumulated to create a percept (typically less than 0.5 sec), a typical neuron may discharge  10 or fewer action potentials, and firing patterns typically differ substantially from Poisson \cite{MochizukiShinomoto}.

 Ultimately, we are interested in comparing the characteristics of the walks with jumps model to findings in the experimental literature. Here however we focus on establishing basic facts about our construction using classical techniques of hyperbolic geometry and geometric group theory. We now give the formal definition.

\begin{definition}\label{D: walkswithjumps}
A \textbf{walk with jumps} is specified by the following \textbf{initial data}:\begin{itemize} 
    \item an \textbf{origin} $o\in\Htwo$, unit tangent vector $\mathbf{v}$ at $o$, a \textbf{walk speed} $s>0$, \textbf{jump angle} $\theta\in(-\pi,\pi)$, and \textbf{jump length} $\ell>0$;\end{itemize}
And the following \textbf{jump data}:\begin{itemize}
    \item a \textbf{duration} $T>0$, sequence $(t_1,\hdots,t_n)$ of \textbf{jump times} ($n\geq 0$), where 
    \[ 0 \le t_1 <\hdots < t_n \le T,\] 
    and a \textbf{burst vector} $(j_1,\hdots,j_n)$ with natural number entries, taken to be $(1,\hdots,1)$ if not otherwise specified. \end{itemize}
 The walk's \textbf{number of jumps} is $N = \sum_{i=1}^n j_i$; we also call it a \textbf{walk with $N$ jumps}.
\end{definition}
    
This data encodes the following set of instructions for movement:

\begin{quote} Beginning at $o$, walk along the geodesic ray in the direction of $\mathbf{v}$ at speed $s$ until time $t_1$. At time $t_1$, instantly jump a distance of $\ell\cdot j_1$ along an axis through the current location at a counterclockwise angle of $\theta$ to the original; then proceed at speed $s$ along a geodesic ray at angle $-\theta$ to the jump axis. Continue until time $t_2$; jump a distance $\ell\cdot j_2$ along an axis at angle $\theta$ to the current one; then proceed right again along an axis at angle $-\theta$ to the jump axis, etc.\end{quote}

Having executed the instructions above and halted at time $T$, our avatar's final location is the \textbf{endpoint} of the walk with jumps. It carries the initial tangent vector $\bv$ with it by parallel transport, finally yielding a tangent vector at the endpoint which we call the \textbf{endvector}.

\begin{remark}\label{burst vectors}
In Definition~\ref{D: walkswithjumps}, we choose to use burst vectors to encode multiple jumps at a single time rather than allowing consecutive jump times to be equal. While these two different approaches are equivalent, using burst vectors will be technically convenient when we write walks with jumps as words in isometries of $\Htwo$ specified by the data $o$, $\textbf{v}$, $s$, $\theta$, $\ell$ and $T$, see Section~\ref{S: first obs}. 
\end{remark}

The central question of this paper is of the extent to which the jump data of a walk with jumps can be recovered from its endpoint $p$, or the finer information $(p,\bw)$, where $\bw\in T_p\mathbb{H}^2$ is the endvector, for a fixed set of initial data. We offer a spectrum of partial answers to this question in Sections \ref{S: same endpts}, \ref{subsec:more diff} and \ref{S: difft endpts}, in both positive and negative directions, reflecting the truism that hyperbolic space is ``Euclidean at small scales and treelike at large scales''. What this should mean for walks with jumps is fleshed out in Section \ref{S: examples for contrast}. We show there that for the analog of a walk with jumps construction in the Euclidean plane, the endpoint is determined solely by the duration and number of jumps. In contrast, for the analogous construction in a tree, any two distinct jump \textit{patterns} produce distinct endpoints.

Section \ref{S: same endpts} develops calculus tools for understanding how perturbations of a walk's jump times affect its endpoint-endvector pair. In Section \ref{S: first obs} we prove Lemma \ref{wwj via isoms}, a basic fact describing how the data of a walk with jumps uniquely prescribes an isometry carrying its origin-initial vector pair $(o,\bv)$ to $(p,\bw)$. Using this we define a differentiable map whose domain is a simplex $T_k$ encoding all walks with $k$ jumps and a fixed duration that share initial data, and whose target space is $\operatorname{SL}_2(\mathbb{R})$, viewed as (a double cover of) the unit tangent bundle of $\mathbb{H}^2$.

The main technical result of Section \ref{S: same endpts}, Proposition \ref{mo buttah} shows that the derivative of this map has full rank on a dense open subset $U$ of $T_k$ for all $k\ge 3$, provided the jump angle $\theta$ is at most $\pi/2$. Several Corollaries follow, including \ref{C: different walks same endpoints} which asserts that for $k>3$, if a walk's sequence of jump times lies in $U$ then there exist arbitrarily small perturbations of this sequence yielding walks with jumps sharing its endpoint and endvector. Corollary \ref{rich behavior} asserts that for $k\ge 3$, the set of endpoints of walks with jumps is the closure of its interior in $\mathbb{H}^2$.

In Section \ref{subsec:more diff} we move beyond perturbations, constructing walks with significantly different jump patterns that still share an endpoint-endvector pair. Proposition \ref{same duration jumps endvector} describes such pairs sharing the same initial data, number of jumps, and duration, but with arbitrarily large gaps between jump times of the first and the second. This construction does rely, however, on there being not too large a gap between the first jump times of the two walks with jumps. In Section \ref{subsec: sepcrit} we give conditions that eliminate this possibility.

But first, in Section \ref{S: difft endpts} we establish some basic structure results by applying the tools of hyperbolic geometry to a class of paths associated to walks with jumps.

\begin{definition}\label{D: walkswithjumpspath}
For a walk with jumps with the data of Definition~\ref{D: walkswithjumps}, the associated \textbf{walk-with-jumps path} is the broken geodesic in $\Htwo$ obtained by joining the origin $o$ to the pre-jump location $p_1$ at time $t_1$; joining $p_1$ to the post-jump location $q_1$ at time $t_1$; then joining $q_1$ to the pre-jump location $p_2$ at time $t_2$ and so on. Its $i$th \textbf{walk segment} joins $q_{i-1}$ to $p_i$ (or $o$ to $p_i$, if $i=1$, or $q_n$ to the endpoint, if $i = n+1$), and its $i$th \textbf{jump segment} joins $p_i$ to $q_i$. We parametrize it continuously on $\left[0,sT + \ell\cdot \sum_{i=1}^n j_i\right]$, mapping $0$ to $o$, by parametrizing each geodesic segment in turn by arclength.\end{definition}

A quick inductive argument shows that the pre- and post-jump locations  $p_i$ and $q_i$ at time $t_i$ occur at parameter values $st_i + \ell\sum_{k=1}^{i-1} j_k$ and $st_i + \ell\sum_{k=1}^i j_k$, respectively, for $i\in\{1,\hdots,n\}$. 

\begin{remark}\label{R: endvector} The endpoints of the walk-with-jumps path associated to a walk with jumps are the origin $o$ and the endpoint of the walk with jumps itself, as defined above. Moreover, its endvector is the outward-pointing unit tangent vector to the final segment of the walk-with-jumps path or, if the final jump time $t_n$ equals $T$, the unit tangent vector at an angle of $-\theta$ from that vector.
\end{remark}

In Section \ref{subsec: no return} we lay the groundwork for our further results by associating two collections of pairwise disjoint geodesics, the \textit{walk} and \textit{jump axes}, to each walk with jumps. The main result of the subsequent Section \ref{subsec: qg}, Proposition \ref{trig qg} asserts that any walk-with-jumps path is a \textit{quasigeodesic}, meaning the hyperbolic distance between any two points on it is well-approximated by their distance along the path (see Definition \ref{D: qg}). This implies that walks with the same duration but different enough numbers of jumps have different endpoints, see Corollary \ref{C: jump simplices}. 

The main result of Section \ref{subsec: sepcrit} shows that a pair of conditions, which we now define, ensures that distinct walks with jumps have different endpoints.

\newcommand\DistToResEps{
For $\epsilon > 0$, two walks with jumps sharing initial data and a duration $T>0$ are \textit{distinct to resolution $\epsilon$} if their sequences of jump times $(s_1,\hdots,s_m)$ and $(t_1,\hdots,t_n)$ satisfy:\begin{enumerate}
    \item for each $i\le \min\{m,n\}$, either $s_i = t_i$ or $|s_i-t_i|>\epsilon$; and
    \item there exists some $i$ such that $s_i\ne t_i$, or $m\ne n$.
\end{enumerate}
For $R_{min}>0$, a walk with jumps has \textit{minimum refractory length $R_{min}$} if its burst vector is $(1,\hdots,1)$---ie.~it has no bursts---and for $i\ne j$, $|t_i-t_j|\ge R_{min}$.}
\begin{definition}\DistToResEps\end{definition}

While these conditions may not appear geometrically natural, they are meaningful in biological contexts in which it is not possible to discern differences at arbitrarily small resolutions, and in which physical constraints prevent a neuron from firing again too soon after it has fired once. At suitable scales, they prohibit distinct walks with jumps from having the same endpoint.

\newcommand\SepCriterion{Given $R_{\min}>0$, the initial data $(o,\bv,s,\theta=\pi/2,\ell)$ of a walk with jumps, and a duration $T>0$, let
\[ \epsilon = \frac{1}{s}\cosh^{-1}\left(\frac{\cosh (sR_{min})+1}{\cosh (sR_{min})-\tanh^2\ell} \right).\]
Any two walks with jumps having the given initial data, duration $T$, and minimum refractory length $R_{min}$ that are distinct to resolution $\epsilon$, have distinct endpoints.
}

\begin{theorem}\label{sep criterion}\SepCriterion\end{theorem}
\theoremstyle{plain}
\newtheorem*{sep criterion thrm}{Theorem \ref{sep criterion}}

Note that for any fixed positive $s$ and $\ell$, $\epsilon$ defined as above can be made arbitrarily small by choosing $R_{\min}$ large enough. We prove a stronger version of Theorem \ref{sep criterion} in Section \ref{subsec: sepcrit} as Theorem \ref{precise sep criterion}. This more precise version also yields Corollary \ref{the ned}, on the existence of trees whose edges are traversed by walk-with-jumps paths.

Section \ref{sec: FurQue} lists further questions, including some that are biologically motivated and others motivated by the hyperbolic-geometric phenomenon called \emph{quasigeodesic stability}, see Proposition \ref{P: stability}.

\subsection*{Acknowledgements} JD and EE gratefully acknowledge support from the Swartz Foundation during the initial phase of this work. JDV is supported by NIH EY07977 and NSF  2014217. We thank H. Anuradha Ekanayake, Arshia Gharagozlou, and Mark Fincher for helpful conversations, and Mark for supplying Figures \ref{F: visualizing spike trains}, \ref{F: one jump}, and \ref{future quadrants}.

    \subsection{Biological Motivation}\label{biology}
    
We are proposing the walks-with-jumps model as a theoretical framework for modeling neural activity and how it evolves in time. Standard approaches assume (often implicitly) that this evolution occurs in a space with an underlying Euclidean geometry; the material here provides a  foundation for replacing this space by a hyperbolic one. 
 
    As is standard, we idealize the activity of a single neuron during a time interval as a sample drawn from a point process, i.e., a sequence of stereotyped electrophysiological events (action potentials) which together constitute a ``spike train``.  
    
     \begin{definition}\label{D: spike train}
    Let $t\in \R^+$ and let $k\in\mathbb{N}$. 
    A \textbf{spike train (over the period [0,t]) with $k$ spikes} is a sequence $t_1 \le t_2 \le t_3 \le t_4 \le \ldots \le t_k$ with each $t_i\in [0,t]$. We say that \textbf{its spikes occur} at the times $t_i$.
    \end{definition}
    
\noindent That is, a neuron producing the above spike train has spikes at the times $t_i$ and is quiescent  over the periods from $t_i$ to $t_{i+1}$ for each $0\le i\le k-1$. In our model, the information represented by the neuron over the time interval $[0,t]$ is fully contained in its set of spike times.
    
    \begin{remark}\label{R:bursts} Some neurons exhibit \mbox{\rm bursts}: sequences of spikes in very short succession.  In the limit of a vanishingly small interspike interval, we can model a burst as a single spike with a magnitude that is a natural number multiple of the stereotypical spike's.\end{remark}
 
       A neuron's activity influences and is influenced by that of others in its network. A central challenge of systems neuroscience is to understand and model how the aggregate activity of a network of such neurons performs cognitive tasks, such as representing the external world or making decisions. Here, we are motivated by considerations about how the activity of individual neurons may contribute to these processes. 
       For example, we can ask the following fundamental question: what way of viewing a spike train captures the ``right'' amount and quality of information processed by the neuron? An intuitive constraint on the answer is the fact that neural activity is a biological process involving some level of noise and imprecision, so for instance a small variation in the exact timing of its spikes should result in a small change in what this activity represents.
       
    An important stepping stone for pursuing this question is the notion of a ``perceptual space`` \cite{ZaidiEtAl} -- a mathematical object capturing the mind's internal representation of a sensory domain, whose structure (e.g. topological, metric, or linear) reflects subjective notions of similarity and difference.

    A classic example is the perceptual space of color \cite{ZaidiEtAl}.  It is well-established that for (non-colorblind) humans, this is a three-parameter space \cite{LeeReview2008}. More specifically, colors may be described in terms of real-valued coordinates along three ``cardinal'' perceptual axes  -- black vs. white, red vs. green, and blue vs. yellow. To the extent that this space can be regarded as Euclidean -- which is clearly an approximation -- there is a reasonably satisfying understanding of how it is represented by neural activity  (we omit many interesting and important physiological details here; see \cite{DerringtonKrauskopfLennie,SolomonLennie}). 
    In brief, light is transformed into neural signals by three types of photoreceptors, each with its own wavelength-dependent sensitivity to light. Subsequently, neural circuitry in the eye and brain recombines these signals in a way that is, to a first approximation, linear. This leads to coordinates which, while not directly corresponding to the three cardinal axes, can be thought of as a different basis set for the same Euclidean space.  With these ingredients, it is straightforward to see how spike trains could represent the perceptual space of color: it suffices to postulate that neurons correspond to axes in a three-dimensional space, and the total number of spikes within a time interval is the coordinate value represented by a neuron.
    
    Note that in the model above for the color vision perceptual space, the useful information content of a neuron's activity is entirely captured by a coarse invariant of its spike train, the total spike number. We take this to reflect that space's relatively simple structure. However, other perceptual spaces appear to be qualitatively different, and the relationship  between their structure and the underlying neural activity is much less clear.
   
   A key motivating example for this paper is the human olfactory system. Here there are several hundred kinds of receptors \cite{HasinBrumshteinReview}, in contrast to the three kinds of receptors relevant to color vision. Furthermore, the perceptual space of odors likely has intrinsically non-Euclidean geometry \cite{ZhouSmithSharpee}. So the idea that spike counts behave like Euclidean coordinates, although sufficient to account for how the perceptual space of color might be represented, is likely inadequate to capture the olfactory perceptual space's complexity.
   
    A simple but important elaboration on this picture is that additional information is contained in the timing of individual spikes. Many experimental studies have shown that this is the case, e.g., \cite{JacobsNirenbergRulingInAndOut,ReichMechlerVictor,RichmondOptican1990}. To formalize this idea, a spike train can be represented as a counting process, i.e., a path in which time without a spike consists of movement in one direction, and a spike event corresponds to a jump in a different direction, see e.g. Figure~\ref{F: visualizing spike trains}. Two spike trains then convey similar messages when their paths are similar. Most often, similarity of spike trains is assumed to be computed by embedding them in a vector space and computing the  Euclidean distance between them \cite{RichmondOptican1990}, or, via a ``spike time'' metric.  The latter is an ``edit-length distance'' \cite{Sellers} given by the minimum ``cost'' to morph one spike train into another by inserting or deleting spikes, or shifting them in time \cite{VictorPurpura1997}.

    However, merely taking spike timing into account does not solve the problem of representing a perceptual space with hyperbolic characteristics. This is because even though the spike time distances are non-Euclidean \cite{AronovVictor2004}, they (as well as distances derived from vector-spaces) obey a superposition property:  using $+$ to denote superposition of spike trains, $D(A,B)=D(A+X,B+X)$. Specifically, inserting a spike at a fixed time to both spike trains will not change the distance between them.

    Thus, to be able to capture the hyperbolic geometry that appears to characterize some perceptual spaces, a further generalization is needed. This is our walk-with-jumps model (Figure~\ref{F: visualizing spike trains}C):  spike trains are represented by a path that evolves in time, but the two motions corresponding to time without a spike, and the presence of a spike, are isometries in  $\Htwo$. Note that, because of the non-commutativity of isometries in  $\Htwo$, the superposition property no longer holds.  Moreover, the endpoint of a path in $\Htwo$ depends not only on the number of spikes, but also on their arrangement in time. We elaborate on this below.

    A separate biological motivation for the walk-with-jumps model is that it provides a new way of viewing the neural substrate of decision-making.  There are currently two main conceptual frameworks for this.  One is that competing populations of neurons accumulate evidence  (i.e., spikes), until one of them reaches a threshold \cite{GoldShadlen}. A second framework views neural population activity as a dynamical system whose evolution in time is determined by interactions of excitatory and inhibitory neural signals. In this view, decision-making corresponds to entry into an attractor's basin, from which there is no return \cite{WangXJ}. Here, we note that if neural activity evolves in $\Htwo$, then a ``point-of-no-return`` property 
    means that such decisions can result merely from the intrinsic geometry of the space in which neural activity evolves. As we show in Theorem \ref{sep criterion}, this behavior is typical of walks with jumps in $\Htwo$: if two walks with jumps  are sufficiently different up to some time point, then future extensions of them cannot intersect.  
    
    Finally, we note that the geometry of $\Htwo$ ensures that the dynamics of the walks-with-jumps model has tree-like features. First, the volume of potential endpoints grows exponentially as a function of the distance from the origin. While the continuous nature of the model will accommodate noise in the exact timing of the spikes, any significant local deviations in the timings of the spikes will lead to walks-with-jumps paths with very different endpoints. This tree-like behavior is natural for modeling the coarse-to-fine characteristics of visual perception  \cite{Hegde2008}, as well as perceptual spaces with semantic content.

\subsection{Comparing models of a single neuron's activity} \label{S: examples for contrast}\
    
   \begin{figure}[ht]
    \centering
    \def\svgwidth{\columnwidth}
    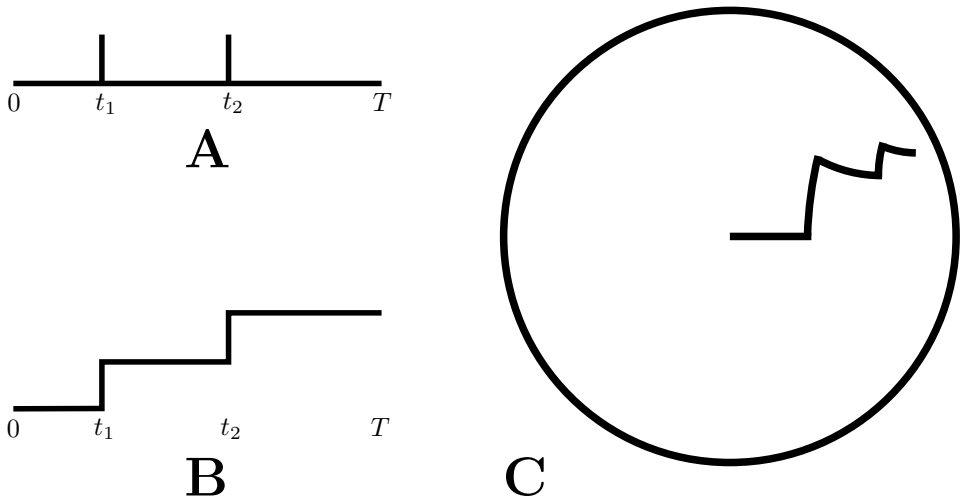
    \caption{This figure shows three different ways of visualizing a spike train.}
    \label{F: visualizing spike trains}
\end{figure}
 
    Figure \ref{F: visualizing spike trains}A is a standard conceptualization of a neuron's activity over a time period $[0,T]$, during which it fires spikes at $t_1 < t_2 < T$ (as defined in \ref{D: spike train}). An equivalent conceptualization is shown in  Figure \ref{F: visualizing spike trains}B:  activity is represented by the counting process $F(t)$, where $F(t) = \int_0^t f(s)\,\mathit{ds}$ and $f(t)$ is a sequence of (unit-mass) delta-functions corresponding to the point process schematized in Figure \ref{F: visualizing spike trains}A.  
    
    This second conceptualization has another interpretation:  the neuron's activity is motion of a point along a path in $\mathbb{R}^2$, in which the point moves in one direction at constant speed when there is no spike, and then abruptly jumps in another (orthogonal) direction at the times that a spike occurs. 
    
    That is, Figure \ref{F: visualizing spike trains}B conceptualizes a neuron's activity as  a ``walk with jumps'', as in Definition \ref{D: walkswithjumps} but in $\mathbb{R}^2$ not $\mathbb{H}^2$, with origin $(0,0)$, unit tangent vector $(1,0)$, walk speed $1$, jump angle $\pi/2$, duration $T$, and jump length $M$. For comparison, Figure \ref{F: visualizing spike trains}C shows the walk with jumps with the same data, but in $\mathbb{H}^2$.
    
In this section we compare and contrast qualitative features of ``walk with jumps`` style constructions in various metric contexts, beginning below with that of $\mathbb{R}^2$.  As we will show, for trajectories in $\mathbb{R}^2$, the endpoint is determined solely by the number of spikes and the time interval.  That is, two spike trains with the same endpoint can have very different firing patterns.  However, in $\mathbb{H}^2$, the arrangement of the spikes also influences the endpoint -- a feature that is appealing from the point of view that in neural circuits, spike timing, as well as spike count, conveys information.  Our overall focus is on the relationship between the endpoint and the path, for example, to what extent do spike trains with the same endpoint necessarily have similar paths?
 
  \begin{example}\label{E: eucl}
  Consider a walk with jumps as in Definition \ref{D: walkswithjumps}, but in $\mathbb{R}^2$ not $\mathbb{H}^2$, with origin $(0,0)$, unit tangent vector $(1,0)$, walk speed $1$, jump angle $\pi/2$, jump length $1$, duration $T$, and jump times $0\leq t_1 \leq t_2 \leq \cdots \leq t_k \leq T$. 
Observe that the corresponding walks with jumps path has endpoint:
\[(T,k) = (t_1,0) + \sum_{i=1}^{k}\left( (t_{i+1}-t_i),1\right) + (T-t_k,0) \]
That is, the path's endpoint is determined entirely by the walk's duration and number of spikes. 
\end{example}

In Example~\ref{E: eucl}, the walk path endpoints do not take into account  the timing of individual spikes in a spike train. In contrast, spike timing does affect path endpoints for walks with jumps in $\mathbb{H}^2$. The result below uses hyperbolic trigonometry to quantify the effect produced by changing the timing of a single spike, when the jump angle is $\pi/2$. In fact we show that if two such walks have a large enough gap in the timing of their first spikes, then they can never have the same endpoint.
    
\begin{proposition}\label{distinct endpoints} Consider two walks with jumps in $\mathbb{H}^2$, sharing initial data $o$, $\mathbf{v}$, $s$, jump angle $\theta = \pi/2$, and jump length $\ell>0$. If each walk has a single jump, and this occurs at time $0$ in the first and at time $T$ in the second, then the distance $d$ between the two walks' endpoints satisfies:
\[ \cosh(d)-1 = (\cosh(\ell)-1)(\cosh(sT)-1)(\cosh(\ell)\cosh(sT)-1).\]
Regardless of how many jumps each has, if the first walk has its first jump at time $0$ and the second has its first jump at time $t_1$ satisfying
\[ \sinh(\ell)\sinh(s\, t_1)\ge 1, \]
then the two walks have distinct endpoints.
\end{proposition}

\begin{figure}[ht]
    \centering
    \def\svgwidth{\columnwidth}
    \scalebox{0.7}{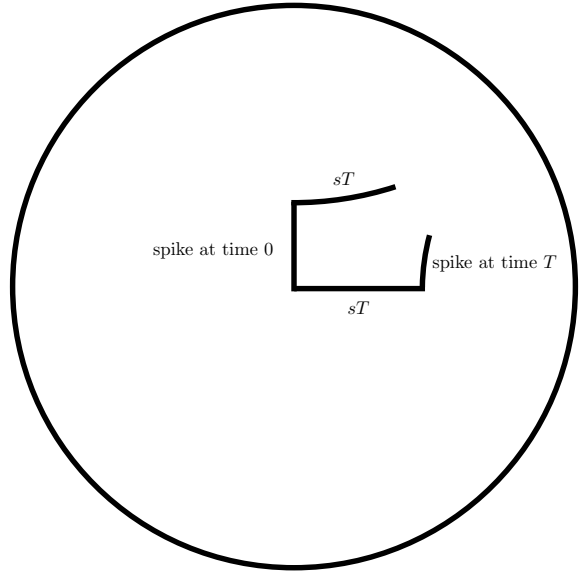}
    \caption{This figure shows two hyperbolic walks with jumps over the same time period, $T,$ with a single jump. In one walk, the jump occurs at the beginning of the time period, and in the other it occurs at the end.}
    \label{F: one jump}
\end{figure}

\begin{proof} The union of walks-with-jumps paths (in the sense of Definition \ref{D: walkswithjumpspath}) for the first two walks considered here (having one jump each) is pictured in Figure \ref{F: one jump}, using the Poincar\'e disk model with $o$ at the origin. The vertical line segment in the Figure is the first walk's jump segment, which has length $\ell$ in $\mathbb{H}^2$. The horizontal line segment is the second walk's walk segment, of length $sT$. The other two segments are each contained in circular arcs meeting the unit circle perpendicularly---the form of any hyperbolic geodesic not containing the origin in the Poincar\'e disk model.

Let $q_1$ be the far endpoint from $o$ of the first walk's jump segment, at the top left corner of the shape made by the walk-with-jumps paths of Figure \ref{F: one jump}, and let $p_1'$ be the far endpoint from $o$ of the second walk's walk segment. The shortest arc joining the first walk's endpoint to the hyperbolic geodesic containing the second walk's walk segment meets that segment at an angle of $\pi/2$ at a point $p_0'$, and therefore defines a \emph{Lambert quadrilateral} $Q$---one with three right angles---whose vertices are $o$, $q_1$, $p_0'$, and the first walk's endpoint.

The first walk's jump segment is one side of $Q$, of length $\ell$, and its walk segment is another, of length $sT$. 
Let $x$ be the length of the side of $Q$ joining $o$ to $p_0'$, $y$ the distance from the first walk's endpoint to $p_0'$, and $\gamma$ the interior angle at the first walk's endpoint. A hyperbolic law of cosines for Lambert quadrilaterals gives:
\begin{align}\label{lambert} \sinh(\ell)\sinh(x) = \cos(\gamma). \end{align}
See Theorem 3.5.9 of \cite{Ratcliffe}. A law of cosines for more general quadrilaterals having right angles at two adjacent vertices, recorded as \cite[Th.~3.5.8]{Ratcliffe}, specializes in the case of $Q$ to give $\cosh(sT) \sin(\gamma) = \cosh(x)$. Using the first equation to replace $\sin(\gamma)$ in the second, and solving for $\sinh(x)$, yields the formula below:
\[ \sinh(x) = \frac{\sinh(sT)}{\sqrt{\cosh^2(sT)\sinh^2(\ell)+1}}. \]
This implies in particular that $x<sT$. Section VI.3.3 of \cite{Fenchel} gives several more trigonometric formulas for these more general quadrilaterals. One with the form of a hyperbolic law of sines specializes to $\sinh(y) = \cosh(sT)\sinh(\ell)$ for $Q$. We apply a final formula from \cite{Fenchel} to another such quadrilateral, this one with vertices at the first and second walks' endpoints, and at $p_0'$ and $p_1'$. It gives:
\[ \cosh(d) = -\sinh (y)\sinh (\ell) + \cosh(y)\cosh(\ell)\cosh(sT-x).\]
The Proposition's first conclusion is obtained by substituting for $x$ and $y$ in this formula using those obtained above, and simplifying.

We now consider the second situation contemplated in the Proposition, in which both walks may have many jumps but the first walk still has its first jump at time $0$. As in the previous case, let $q_1$ be the far endpoint from $o$ of the first walk's jump segment, and let $p_1'$ be the far endpoint of the second walk's first walk segment. The distance from $o$ to $q_1'$ is again $\ell$, and in this case from $o$ to $p_1'$ is $st_1$.

The geodesic ray from $q_1$ that contains the first walk's next walk segment is at a right angle to its first jump segment; and likewise, the geodesic ray containing the second walk's first jump segment is at right angles to its first walk segment. If these rays intersect, then their subsegments joining $q_1$ and $p_1'$ to the point of intersection form a Lambert quadrilateral, together with the first walk's first jump segment and the second walk's first walk segment. If so then by the law of cosines for Lambert quadrilaterals, as in (\ref{lambert}) above, we would have $\sinh(\ell)\sinh(s t_1) = \cos(\gamma)$, where $\gamma$ is the angle at which the rays meet.

Thus if $\sinh(\ell)\sinh(s t_1)\geq 1$ then the two geodesic rays do not intersect. In this case we can deduce that the walks will not share endpoints, since the first walk's endpoint is on the other side from $o$ of the geodesic containing the ray from $q_1$, and the second walk's endpoint is on the other side from $o$ of the geodesic containing the ray from $p_1'$ . (We will prove this in Lemma \ref{quadrants} below.)
\end{proof}

\begin{figure}
    \centering
    \includegraphics[width = 5cm, height = 5cm]{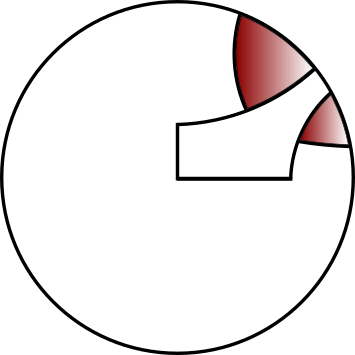}
    \caption{Disjoint future quadrants.}
    \label{future quadrants}
\end{figure}

The train of thought that finishes off the proof above motivates the notion of ``future quadrant'', an intersection of half-planes associated to an initial segment of a walk with jumps that is guaranteed to contain the walk's endpoint. This is formally defined in Definition \ref{future quad}. See Figure \ref{future quadrants} for the picture relevant to the proof of Proposition \ref{distinct endpoints}.

The behavior described in Proposition \ref{distinct endpoints} is reminiscent of paths in a tree, which, once they diverge, can never come back together. Another motivating example for the walk with jumps construction is a strategy for encoding a(n approximation to a) spike train by a walk in a tree, as we now describe.

Given a spike train $S$ over a time period $t$, we can encode a discrete approximation to $S$ in a binary sequence $w$ of length $n$ as follows: for each $1\le i \le n$, if $S$ has a spike in $[\left(\frac{i-1}n\right) t, \left(\frac{i}n\right) t]$, then $i$th entry of $w$ is $1$. If there is no spike in $[\left(\frac{i-1}n\right) t, \left(\frac{i}n\right) t]$, then the $i$th entry of $w$ is $0$. 
We can view these discrete approximations of spike trains geometrically by embedding them in a binary tree:

\begin{example}\label{E: binary tree}
Let $\mathcal{T}$ be a binary tree with root $p_0$. A binary sequence of length $n$, $w = w_1w_2\ldots w_n$ where $w_i\in \{0,1\}$, encodes a path of length $n$ from $p_0$ to a vertex of $\mathcal{T}$ as follows: at the root proceed left if $w_1=1$ and right if $w_1=0$. At the $(i-1)$th vertex of the path, proceed left if $w_i=1$ or right if $w_i=0$. 

If $\gamma,\sigma$ are paths (encoded as binary sequences) of length $m,n$ respectively, then the distance between their endpoints can be computed as follows: if $\gamma = \gamma_1\gamma_2\ldots \gamma_m$ and $\sigma = \sigma_1\sigma_2\ldots\sigma_n$, let $d$ be the length of the maximal common prefix of $\gamma$ and $\sigma$. Then the distance between the endpoints of $\sigma$ and $\gamma$ is equal to $(m-d)+(n-d)$. 
\end{example}

Walks-with-jumps paths in $\Htwo$ have geometry similar to the paths encoded by binary sequences in Example~\ref{E: binary tree} with some distinct advantages. 
First, walks-with-jumps paths can model continuous processes. In the tree $T$, the binary sequences 
\[LRLLLLLLLLLLLLLL \qquad \text{ and } \qquad LLLLLLLLLLLLLLLL\]
are quite similar, but their endpoints will be far apart. 
The walks-with-jumps model has the feature that  non-identical paths that can track each other closely if the event times are similar-- behavior that is biologically appealing, since the underlying biophysics of neuronal activity necessarily places a limit on the precision of spike times. This feature more effectively accounts for noise that appears in data.

\section{Distinct walks with identical endpoints}\label{S: same endpts}

In the final section we will produce a fairly general set of conditions in which two walks with jumps are guaranteed to have distinct endpoints, suggesting that to some extent, the endpoint ``encodes'' the times of the jumps. But first, here, we show that this encoding is not fully faithful. In Section \ref{subsec:perturb} we show this  via a general but non-constructive argument that implies coincident endpoints for walks with jumps whose jump times are not very different; following that, in Section \ref{subsec:more diff}, we create explicit  examples of walks with jumps that have distinct jump times but reach the same endpoint, and the jump times of these walks can differ substantially.

\input{NewProofsForSection2}

\subsection{Perturbing jump times without changing the endpoint}\label{subsec:perturb}

We begin with a standard result that combines two assertions: that the isometry group of $\mathbb{H}^2$ acts transitively on it, and that an isometry is determined by where it sends a single point and tangent vector at that point. Below $\mathit{UT}\,\mathbb{H}^2$ is the unit tangent bundle of $\mathbb{H}^2$, consisting of pairs $(p,\mathbf{v})$ for $p\in\mathbb{H}^2$ and $\bv$ a unit-length tangent vector at $p$, and $\operatorname{SL}_2(\mathbb{R})$ is the group of $2\times 2$ real matrices with determinant 1.

\begin{proposition}\label{P: unit tangent bundle determines isometry} For a fixed $o\in\mathbb{H}^2$ and unit tangent vector $\mathbf{v}$ at $o$, there is a diffeomorphism $\operatorname{PSL}_2(\mathbb{R})\to \mathit{UT}\,\mathbb{H}^2$ that sends $w\in\operatorname{PSL}_2(\mathbb{R})$ to the pair $(w(o),w_*(\mathbf{v}))$ in $\mathit{UT}\,\mathbb{H}^2$, for $\operatorname{PSL}_2(\mathbb{R}) = \operatorname{SL}_2(\mathbb{R})/\{\pm I\}$.
\end{proposition}

\begin{proof} Taking $\mathbb{H}^2$ as the set of $z\in\mathbb{C}$ with positive imaginary part, $\operatorname{SL}_2(\mathbb{R})$ acts on it via M\``obius transformations: for $z\in\mathbb{H}^2$ and $a,b,c,d\in\mathbb{R}$ such that $ad-bc = 1$,
\[ \begin{pmatrix} a & b \\ c & d \end{pmatrix}.z = \frac{az+b}{cz+d}. \]
This action extends to a smooth $\operatorname{SL}_2(\mathbb{R})$-action on the tangent bundle $\mathbb{H}^2\times \mathbb{C}$ by taking $\left(\begin{smallmatrix} a & b \\ c & d \end{smallmatrix}\right)$ to act on the tangent vector $v\in\mathbb{C}$ at $z\in\mathbb{H}^2$ by multiplication by its derivative $1/(cz+d)^2$. This action preserves the hyperbolic Riemannian metric and hence restricts to an action on
\[ \mathit{UT}\mathbb{H}^2 = \{ (z,v) \in\mathbb{H}^2\times\mathbb{C}\,|\, |v| = \Im z \} \]

It is an exercise to check that the $\operatorname{SL}_2(\mathbb{R})$-action on $\mathbb{H}^2$ is transitive, and that the stabilizer of $i$ is
\[ \operatorname{SO}(2) = \left\{\begin{pmatrix} \cos\theta & -\sin\theta \\ \sin\theta & \cos\theta \end{pmatrix} \right\}. \]
Such a matrix acts on a tangent vector $v$ at $i$ by multiplication by $1/(i\sin\theta + \cos\theta)^2 = e^{-2i\theta}$, ie.~as rotation by $-2\theta$. Therefore the extended action on $\mathit{UT}\mathbb{H}^2$ is also transitive, i.e., ~$\mathit{UT}\mathbb{H}^2$ is a \textit{homogeneous space} of $\operatorname{SL}_2(\mathbb{R})$. Furthermore for any unit complex number $v$, regarded as a tangent vector to $\mathbb{H}^2$ at $i$, the stabilizer in $\operatorname{SL}_2(\mathbb{R})$ of $(i,v)$ is the set
\[ \left\{ \begin{pmatrix} \cos(n\pi) & -\sin(n\pi) \\ \sin(n\pi) & \cos(n\pi)\end{pmatrix}\right\} = \{ \pm I \}. \] 
Proposition~\ref{P: unit tangent bundle determines isometry}  now follows directly from the characterization of homogeneous spaces, see eg.~\cite[Theorem 21.18]{LeeSmoothMflds}. This asserts that for a Lie group $G$ acting on a homogeneous space $X$, and any fixed $x_0\in X$, the smooth map $G\to X$ given by $g\mapsto g.x_0$ induces a diffeomorphism $G/H\to X$, where $H$ is the stabilizer of $x_0$ in $G$.\end{proof}

\begin{definition}\label{D: walk simplices}
For $T\in\reals^+$ and $k\in\mathbb{N}$, let 
\[ T_k:=\{(s_1,\ldots,s_k):\, 0 < s_1<s_2<s_3<\ldots< s_k < T\}\subset \reals^k \]
and 
\[ \overline{T}_k:=\{(s_1,\ldots,s_k):\, 0 \leq s_1\le s_2\le s_3\le \ldots\le s_k \le T\}\subset \reals^k. \]\end{definition}

\begin{remark} $\overline{T}_k$ is a simplex in $\mathbb{R}^k$ which we view as encoding all possible walks with $k$ jumps and duration $T$. A tuple $\mathbf{s} = (s_1,\hdots,s_k)$ in $\overline{T}_k$ encodes the same jump data as the pair $(t_1,\hdots,t_n)$, $(j_1,\hdots,j_n)$ from Definition \ref{D: walkswithjumps}, where $n\leq k$ is the number of \mbox{\rm distinct} $s_i$ and for each $l\le n$, $t_l$ is a distinct such $s_i$, and $j_l$ is the number of $s_i$ equal to $t_l$.

$T_k$ is the interior of $\overline{T}_k$, an open subset of $\mathbb{R}^k$ consisting of paths that have burst vector $(1,1,\hdots,1)$.\end{remark}

\begin{lemma}\label{L: simplex to matrix} Given the initial data $o,\textbf{v},s,\theta,\ell$ of a walk with jumps as in Definition \ref{D: walkswithjumps}, and $k\in\mathbb{N}$, $T>0$: choose matrices representing the isometries $f$ and $(g_t)_{t\in\mathbb{R}}$ of Lemma \ref{wwj via isoms first half} so that $f$ has positive trace and $g_0 = I$. Continuing to refer to these matrices as $f$ and $g_t$, define $\Omega\co \mathbb{R}^k \to \operatorname{SL}_2(\mathbb{R})$ by 
\[ \Omega(\mathbf{s}) =  g_{s_1}f g_{s_2-s_1}f\cdots g_{s_{k}-s_{k-1}}fg_{T-s_{k}}, \]
for $\mathbf{s} = (s_1,\hdots,s_k)$. This map is continuously differentiable.\end{lemma}

\begin{remark} If $\mathbf{s}\in \overline{T}_k$, for $\overline{T}_k$ from Definition \ref{D: walk simplices}, then a computation shows that $\Omega(\mathbf{s})\in\operatorname{SL}_2(\mathbb{R})$ represents the isometry $w$ from Definition \ref{D: walk and jump axes}.\end{remark}

\begin{proof} We will show that $\Omega$ has continuous partial derivatives, from which continuous differentiability will follow. To compute its $i$th partial derivative at $\mathbf{s} = (s_1,\hdots,s_k)\in T_k$, we write $w\doteq\Omega(\mathbf{s})$ as $L_i\,f\,R_i$, for
\begin{align}\label{lefty righty} 
L_i = g_{s_1}fg_{s_2-s_1}\cdots f g_{s_i-s_{i-1}}\quad \mbox{and}\quad R_i = g_{s_{i+1}-s_i}f\cdots fg_{T-s_k}. \end{align}
For small $t>0$, we then have 
\[ \Omega(s_1,\hdots,s_{i-1},s_i+t,s_{i+1},\hdots,s_k) = L_i\,g_tfg_{-t}\, R_i \]
We may thus compute $\frac{\partial}{\partial s_i}\Omega(\mathbf{s}) = \left.\frac{d}{dt}\right\vert_{t=0}(D_i\circ C)(t)$ with the chain rule, where $D_i\co\operatorname{SL}_2(\mathbb{R})\to\operatorname{SL}_2(\mathbb{R})$ is given by $D_i(X) = L_i\,X\,R_i$ and $C(t) = g_{t}\, f\, g_{-t}$.

We compute $C'(0)$ using the ``product rule for matrix multiplication'': for differentiable families $(A_t)$ and $(B_t)$ of $2\times 2$ matrices, \[ \left.\frac{d}{dt}\right\vert_{t=0}(A_tB_t) = A_0\left[\left.\frac{d}{dt}\right\vert_{t=0} B_t\right] + \left[\left.\frac{d}{dt}\right\vert_{t=0} A_t\right] B_0.\] 
This yields $C'(0) = \mathfrak{g}f - f\mathfrak{g}$, where $\mathfrak{g}= \left.\frac{d}{dt}\right\vert_{t=0}g_t\in\mathfrak{sl}_2(\mathbb{R})$, since $g_0 = I$. Since the action of $D_i$ on $\operatorname{SL}_2(\mathbb{R})$ restricts its action as a linear map of $M_2(\mathbb{R})$, the vector space of $2\times 2$ matrices, $D_i$ is its own derivative and we have
\[ \frac{\partial}{\partial s_i}\Omega(\mathbf{s}) = L_i(\mathfrak{g}f - f\mathfrak{g})R_i. \]
The only $\mathbf{s}$-dependence above lies in the left- and right-multipliers $L_i$ and $R_i$. These vary continuously with $\mathbf{s}$, by repeated applications of the continuity of the multiplication map on $\operatorname{SL}_2(\mathbb{R})$, and for the same reason so does $\partial \Omega/\partial s_i$.\end{proof}

\begin{remark}\label{embark} The partial derivatives computed above belong to the tangent space of $\operatorname{SL}_2(\mathbb{R})$ at $w$. Translating each back to the identity using left-multiplication by $w^{-1}$ yields $R_i^{-1}f^{-1}(\mathfrak{g}f - f\mathfrak{g})R_i = R_i^{-1}\left(f^{-1}\mathfrak{g}f - \mathfrak{g}\right) R_i$.
\[ f^{-1}(\mathfrak{g}f - f\mathfrak{g}) = 2s\sin\theta\sinh\ell
    \begin{pmatrix} -\cos\theta\sinh\ell & \cosh\ell+\sin\theta\sinh\ell \\ \cosh\ell + \sin\theta\sinh\ell & \cos\theta\sinh\ell \end{pmatrix} \]
\end{remark}

\begin{cor}\label{cont diff} For data $o,\textbf{v},s,\theta,\ell$, $k$ and $T$, the map $\psi_k\co T_k\to \mathit{UT}\mathbb{H}^2$ given by composing the diffeomorphism from Proposition \ref{P: unit tangent bundle determines isometry} with $\Omega$ from Lemma \ref{L: simplex to matrix} is continuously differentiable.\end{cor} 

\begin{remark}\label{psi_k interp} The map $\psi_k$ defined above takes a sequence of jump times $\mathbf{s}\in T_k$ to the endpoint-endvector pair of the walk with jumps determined by the prescribed initial data, $\mathbf{s}$, and duration $T$. In particular, the image of $\psi_k$ is the collection of endpoints of all walks with $k$ jumps, duration $T$, and initial data specified above.\end{remark}

\begin{proposition}\label{P: map to bundle not inj}
For $k>3$ and $T_k\subset\mathbb{R}^k$ as in Definition \ref{D: walk simplices}, the map $\psi_k:T_k\to \mathit{UT}\mathbb{H}^2$ defined in Corollary \ref{cont diff} is not injective on any open subset $U\subseteq T_k$. That is: for any such $U$ there exist distinct walks with jumps, sharing initial data and each with $k$ jumps and duration $T$, each of whose sequence of jump times is contained in $U$, having the same endpoint-endvector pair.
\end{proposition}

\begin{proof}
Since $\mathit{UT}\mathbb{H}^2$ is a $3$-manifold we may assume that $\psi_k|_U$ maps to $\mathbb{R}^3$, by replacing $U$ with the smaller open set $U\cap\psi_k^{-1}(V)$ for a chart open set $V\subset \mathit{UT}\mathbb{H}^2$ intersecting $\psi_k(U)$, and post-composing with the chart map. If $\psi_k|_U$ were injective then by Brouwer's Invariance of Domain Theorem, $\psi|_U$ would be a homeomorphism between $U$ and $\psi(U)\subset\mathbb{R}^3$. 
However this contradicts the fact, a standard exercise in algebraic topology (cf.~\cite[Th.~2.26]{Hatcher}), that no open subset of $\reals^k$ is homeomorphic to an open subset of $\reals^3$ when $k>3$.  
\end{proof}

Proposition \ref{P: map to bundle not inj} uses only the continuity of the map $\psi_k$ to show that it is non-injective \textit{near every point of $T_k$, at arbitrarily small scales}. Interpreted in terms of walks with jumps, this implies that for any given walk with at least three jumps, one can find distinct others which each differ from the original -- hence also from each other -- by an arbitrarily small perturbation of jump times, and have the same endpoint-endvector pair as each other (but not necessarily as the original walk with jumps). The next result gives a  calculus tool for obtaining still finer information.

\newcommand\bs{\mathbf{s}}

\begin{proposition}\label{mo buttah} Fix $k\in\mathbb{N}$ at least $3$, the initial data $o$, $\textbf{v}$, $s>0$, $\theta\in(0,\pi/2]$, $\ell>0$ of a walk with jumps, and a duration $T>0$, and let $T_k\subset\mathbb{R}^k$ be as in Definition \ref{D: walk simplices}. The set $U\subset T_k$ on which the derivative of the map $\psi_k\co T_k\to\mathit{UT}\mathbb{H}^2$ from Corollary \ref{cont diff} has full rank is dense in $T_k$ and open. 
\end{proposition}

\begin{proof} 
Using the upper half-plane model we may assume without loss of generality that $o = i$ and $\mathbf{v} = (1,0)$. In this case the isometries $(g_t)_{t\in\mathbb{R}}$ and $f$ of Lemma \ref{wwj via isoms first half} are represented by matrices
\[ g_t = \begin{pmatrix} \cosh(st) & \sinh(st) \\ \sinh(st) & \cosh(st) \end{pmatrix}\ \mbox{and}\  
 f = \begin{pmatrix} \cosh \ell + \sin\theta \sinh\ell & \cos\theta\sinh\ell \\ \cos\theta\sinh\ell & \cosh\ell - \sin\theta\sinh\ell \end{pmatrix} \]
 Note that $g_0$ is the identity and $f$ has positive trace, as required in Lemma \ref{L: simplex to matrix}. For $\mathfrak{g} = \left.\frac{d}{dt}\right\vert_{t=0}g_t\in\mathfrak{sl}_2(\mathbb{R})$, a computation now yields that $\mathfrak{g} = s\left(\begin{smallmatrix} 0 & 1 \\ 1 & 0 \end{smallmatrix}\right)$, and hence that
 \[ \mathfrak{v}\doteq\mathfrak{g}f - f\mathfrak{g} = 2s\sin\theta\sinh\ell \begin{pmatrix} 0 & -1 \\ 1 & 0 \end{pmatrix}. \]
 This matrix features in the description of $\frac{\partial}{\partial s_i}\Omega(\mathbf{s})$ in the proof of Lemma \ref{L: simplex to matrix}.

\begin{claim} For any $\bs= (s_1,\hdots, s_k)\in \overline{T}_k$, and any $i$, $\frac{\partial}{\partial s_i}\Omega(\mathbf{s})\ne \mathbf{0}$.
\end{claim}

\begin{proof}[Proof of Claim] Appealing to Remark \ref{embark}, we find that the claim holds if and only if
$f^{-1}\mathfrak{v}\ne\mathbf{0}$. This in turn can be easily verified using the descriptions of $f$ and of $\mathfrak{v}$ recorded immediately above, together with the hypotheses on $s$, $\theta$ and $\ell$ implying that all scale factors of $\mathfrak{v}$ are positive.\end{proof}

\begin{claim} For any $\bs = (s_1,\hdots, s_k)\in \overline{T}_k$, and any $i < k$, $\frac{\partial}{\partial s_i}\Omega(\mathbf{s})$ and $\frac{\partial}{\partial s_{i+1}}\Omega(\mathbf{s})$ are linearly independent.\end{claim}

\begin{proof}[Proof of claim] These two matrices are linearly \textit{dependent} if and only if one is a scalar multiple of the other; ie, using the proof of Lemma \ref{L: simplex to matrix}, if and only if
\[ k\, L_i\,\mathfrak{v}\,R_i = L_{i+1}\,\mathfrak{v}\,R_{i+1}\quad \Leftrightarrow\quad k\,\mathfrak{v} (R_iR_{i+1}^{-1}) = (L_i^{-1}L_{i+1})\mathfrak{v}\]
for some $k\in\mathbb{R}$, where $L_i$ and $R_i$ are as described in (\ref{lefty righty}). Using this description, the equality above is equivalent to the one below, defining $\delta_i = s_{i+1} - s_i$:
\[ k\,\mathfrak{v}\, (g_{\delta_i}f) = (g_{\delta_i}f)\,\mathfrak{v}. \]
If we write $g_{\delta_i}f= \left(\begin{smallmatrix} a & b \\ c & d \end{smallmatrix}\right)$ then since $\mathfrak{v}$ is a scalar multiple of $\left(\begin{smallmatrix} 0 & -1 \\ 1 & 0\end{smallmatrix}\right)$, the above equality is equivalent to following equations simultaneously holding:
\[ -kc = b,\quad kd = a,\quad ka=d,\quad -kb = c. \]
Noting that we have $ad-bc = 1$, it follows that if the above all hold then $k = 1$, $a = d$, and $b = -c$. Performing the computations, we find that
\[ a = d \quad \Leftrightarrow \quad \sin\theta\cosh(s\delta_i)\sinh\ell = 0 \]
But this does not hold for $\theta\in(0,\pi)$ and $\ell>0$, regardless of whether $\delta_i = 0$.
\end{proof}

\begin{claim}\label{rank 3} For any fixed $i$, $1< i < k$, there is an open dense set of $\bs\in \overline{T}_k$ such that $\frac{\partial}{\partial s_{i-1}}\Omega(\mathbf{s})$, $\frac{\partial}{\partial s_i}\Omega(\mathbf{s})$, and $\frac{\partial}{\partial s_{i+1}}\Omega(\mathbf{s})$ are linearly independent.\end{claim}

\begin{proof}[Proof of Claim] For arbitrary $\bs = (s_1,\hdots, s_k)\in \overline{T}_k$, taking $L_i$ and $R_i$ (and $L_{i\pm 1}$, $R_{i\pm1}$) as in (\ref{lefty righty}), the proof of Lemma \ref{L: simplex to matrix} gives
\[ \frac{\partial}{\partial s_{i-1}}\Omega(\mathbf{s}) = L_{i-1}\mathfrak{v} R_{i-1},\quad 
  \frac{\partial}{\partial s_{i}}\Omega(\mathbf{s}) = L_{i}\mathfrak{v} R_{i},\ \ \mbox{and}\ \ 
  \frac{\partial}{\partial s_{i+1}}\Omega(\mathbf{s}) = L_{i+1}\mathfrak{v} R_{i+1}.\]
Multiplying the entire collection on the left by $L_{i}^{-1}$ and on the right by $R_{i}^{-1}$, we find that the linear independence of this collection is equivalent to that of:
\[ \mathcal{B}_0 \doteq \left\{\ (L_{i}^{-1}L_{i-1})\,\mathfrak{v}\, R_{i-1}R_{i}^{-1},\ \ \mathfrak{v},\ \ (L_{i}^{-1}L_{i+1})\,\mathfrak{v}\,(R_{i+1}R_{i}^{-1}) \ \right\}. \]
Taking $\delta_i = s_{i+1}-s_i$ and $\delta_{i-1} = s_i - s_{i-1}$, we further have:
\[ L_{i-1}^{-1}L_i = g_{\delta_{i-1}} f = R_{i-1}R_i^{-1},\ \mbox{and}\ 
    L_{i}^{-1}L_{i+1} = g_{\delta_i} f = R_{i}R_{i+1}^{-1}. \]
Writing $g_{\delta_{i-1}}f$ as $A_{i-1}$ and $g_{\delta_i}f$ as $A_i$, the collection in question has the form:
\[ \mathcal{B}_0 = \left\{\ A_{i-1}^{-1}\mathfrak{v}\,A_{i-1},\ \mathfrak{v},\ 
A_i\,\mathfrak{v}\,A_{i}^{-1} \right\} \]
Taking $A_{i-1} = \left(\begin{smallmatrix} a & b \\ c & d \end{smallmatrix}\right)$ and $A_i =\left(\begin{smallmatrix} x & y \\ z & w \end{smallmatrix}\right)$, and dividing out the common scale factor $2s\sin\theta\sinh\ell>0$ yields:
\[ \mathcal{B} = \left\{  \begin{pmatrix} -ab-cd & -b^2-d^2 \\ a^2+c^2 & ab+cd \end{pmatrix}, \begin{pmatrix} 0 & -1 \\ 1 & 0 \end{pmatrix}, \begin{pmatrix} xz+yw & -x^2-y^2 \\ z^2+w^2 & -xz-yw \end{pmatrix} \right\}. \]
If there existed $\alpha,\beta,\gamma\in\mathbb{R}$ such that the linear combination
\[ \alpha\begin{pmatrix} -ab-cd & -b^2-d^2 \\ a^2+c^2 & ab+cd \end{pmatrix} + \beta \begin{pmatrix} 0 & -1 \\ 1 & 0 \end{pmatrix} + \gamma \begin{pmatrix} xz+yw & -x^2-y^2 \\ z^2+w^2 & -xz-yw \end{pmatrix} = \begin{pmatrix} 0 & 0 \\ 0 & 0 \end{pmatrix},\]
then by considering upper-left entries we find that 
\[ \alpha(ab+cd) = \gamma(xz+yw), \]
and by considering lower-left and upper-right entries that
\[ \alpha(a^2+c^2) + \gamma(z^2+w^2) = -\beta = \alpha(b^2+d^2) + \gamma(x^2+y^2). \]
Eliminating variables we thus find that either $\alpha = \beta = \gamma = 0$ or
\begin{align}\label{trouble}
    \frac{xz+yw}{x^2+y^2-z^2-w^2} = \frac{ab+cd}{a^2-b^2+c^2-d^2} 
\end{align}
The collection $\mathcal{B}$ is thus linearly \emph{dependent} if and only if equation (\ref{trouble}) holds. Note that its left-hand side depends only on $\delta_i = s_{i+1}-s_i$, whereas the right depends only on $\delta_{i-1} = s_i-s_{i-1}$, among quantities that vary with $\bs$. In particular, $\frac{\partial}{\partial s_{i+1}} a = \frac{\partial}{\partial s_{i+1}} b = \frac{\partial}{\partial s_{i+1}} c = \frac{\partial}{\partial s_{i+1}} d = 0$. We claim that the corresponding partial derivative of the left-hand side is positive.

This follows from direct computation. Taking $S_i = \sinh(s\delta_i)$ and $C_i = \cosh(s\delta_i)$, $S_\ell = \sinh\ell$ and $C_{\ell} = \cosh\ell$, and $s_\theta = \sin\theta$ and $c_\theta = \cos\theta$, we have:
\begin{align*}
    A_i & = \begin{pmatrix} x & y \\ z & w \end{pmatrix} =  \begin{pmatrix} C_iC_\ell + s_\theta C_i S_\ell + c_\theta S_i S_\ell & S_i C_\ell - s_\theta S_i S_\ell + c_\theta C_i S_\ell \\ S_i C_\ell + s_\theta S_i S_\ell + c_\theta C_i S_\ell & C_iC_\ell - s_\theta C_i S_\ell + c_\theta S_i S_\ell \end{pmatrix}
\end{align*}
Further computation now yields\begin{align*}
 xz+yw & = 2S_iC_i(S_{\ell}^2+C_{\ell}^2) + 2c_{\theta}(S_i^2+C_i^2)S_{\ell}C_{\ell} \\
    & = \sinh(2s\delta_i)C_{2\ell} + c_{\theta}\cosh(2s\delta_i)S_{2\ell} \\
 x^2 + y^2 - z^2 - w^2 & = 4s_\theta S_\ell C_\ell = 2s_{\theta}S_{2\ell},
\end{align*}
where $C_{2\ell} = \cosh(2\ell)$ and $S_{2\ell} = \sinh(2\ell)$. The hypotheses $\theta\in(0,\pi/2]$ ensures that $c_{\theta}\ge0$,  and $S_i>0$ since $\bs\in T_k$. Thus using that $\delta_i = s_{i+1}-s_i$ and the computation above, we find that $\frac{\partial}{\partial s_{i+1}}(xz+yw) > 0$. Again by the computation above, the denominator $x^2+y^2-z^2-w^2$ of the left side of (\ref{trouble}) does not depend on $\delta_i$; hence the left-hand side of (\ref{trouble}) increases with $s_{i+1}$ as claimed.

But this further implies Claim \ref{rank 3}, since at any $\bs = (s_1,\hdots,s_k)\in T_k$ where equation (\ref{trouble}) holds, increasing $s_{i+1}$ while holding $s_i$ and $s_{i-1}$ constant produces points of $T_k$ arbitrarily close to $\bs$ where (\ref{trouble}) does not hold. Therefore the set of $\bs$ where $\frac{\partial}{\partial s_{i-1}}\Omega(\mathbf{s})$, $\frac{\partial}{\partial s_i}\Omega(\mathbf{s})$, and $\frac{\partial}{\partial s_{i+1}}\Omega(\mathbf{s})$ are linearly independent is dense. And it is also open, being the complement of the solution set to (\ref{trouble}).
\end{proof}

 Given that $U$ contains the open dense set described in Claim \ref{rank 3} for any fixed $i$, the claim immediately implies the result.
 \end{proof}

The next example shows that the set $U$ of Proposition \ref{mo buttah} where the derivative of $\psi_k$ has full rank may be a proper subset of $T_k$.

\begin{example}
    Let us now consider a walk with jumps satisfying the hypotheses of Proposition \ref{mo buttah}, specialize to the case that $\theta = \pi/2$, and consider the equation (\ref{trouble}) in this case. The computations below that equation specialize to:\begin{align*}
        xz + yw & = \sinh(2s\delta_i)\cosh(2\ell) \\
        x^2+y^2-z^2-w^2 & = 2\sinh(2\ell) \\
        ab+cd & = \sinh(2s\delta_{i-1}) \\
        a^2-b^2+c^2-d^2 & = 2\cosh(2s\delta_{i-1})\sinh(2\ell)
    \end{align*}
Setting the appropriate products equal to each other as in (\ref{trouble}), and rearranging terms, we thus obtain the following equation:
\[ \sinh(2s\delta_i) = \frac{\tanh(s\delta_{i-1})}{\cosh(2\ell)}. \]
This equation is satisfied if and only if $\frac{\partial}{\partial s_{i+1}}\Omega(\mathbf{s})$ lies in the span of 
$\frac{\partial}{\partial s_{i-1}}\Omega(\mathbf{s})$ and $\frac{\partial}{\partial s_i}\Omega(\mathbf{s})$. 

We may regard it as recursively prescribing the value of $\delta_i$, given $\delta_{i-1}$. A given value of $\delta_1$ therefore uniquely determines the set of values $\delta_i$, $i>1$, such that the span of all $\frac{\partial}{\partial s_i}\Omega(\mathbf{s})$ is two-dimensional.
\end{example}

\begin{remark} We could ask whether two walks with \mbox{\rm different} durations can ever have the same endpoint. This may be less biologically relevant.\end{remark}

We close this subsection by recording some consequences of the technical Proposition \ref{mo buttah} for endpoints and endvectors of walks with jumps.

\begin{cor}\label{C: different walks same endpoints} Fix $k\in\mathbb{N}$ greater than $3$, the initial data $o$, $\textbf{v}$, $s>0$, $\theta\in(0,\pi/2]$, $\ell>0$ of a walk with jumps, and a duration $T>0$, and let $T_k\subset\mathbb{R}^k$ parametrize the set of walks with $k$ jumps and the given initial data as in Definition \ref{D: walk simplices}. For any walk with jumps having the given initial data and duration, and whose sequence of jump times lies in the dense open set $U\subset T_k$ of Proposition \ref{mo buttah}, there exist arbitrarily small perturbations of its sequence of jump times yielding walks with jumps sharing its endpoint and endvector.
\end{cor}

\begin{proof} Recall that $U$ is defined to be the set of $\bs\in T_k$ at which the derivative of the map $\psi_k\co T_k\to\mathit{UT}\mathbb{H}^2$ from Corollary \ref{cont diff} has full rank, meaning rank three 
since $\mathit{UT}(\mathbb{H}^2)$ is three-dimensional. Since $\psi_k$ is continuously differentiable, it is a standard consequence of the Implicit Function Theorem that for any $\bs \in U$, the level set of $\psi_k$ containing $\bs$ intersects $U$ in a submanifold of dimension $k-3$. Recalling the interpretation of $\psi_k$ from Remark \ref{psi_k interp}, we see that the small perturbations of the present Corollary's statement lie on this submanifold, which has positive dimension since by hypothesis $k>3$.
\end{proof}

\begin{cor}\label{rich behavior} For any fixed $k\ge3 \in\mathbb{N}$, initial data $o$, $\textbf{v}$, $s>0$, $\theta\in(0,\pi/2]$, $\ell>0$ of a walk with jumps, and any duration $T>0$, the set of endpoints of walks with jumps in $\mathbb{H}^2$ having the given initial data, $k$ jumps, and duration $T$ is the closure of its interior. In particular, it has non-empty interior.
\end{cor}

\begin{proof} The set of endpoints of walks with jumps is the image of the simplex $T_k$ of Definition \ref{D: walk simplices} under the composition of the bundle projection  $\mathit{UT}\mathbb{H}^2\to\mathbb{H}^2$ with the map $\psi_k\co T_k\to\mathit{UT}\mathbb{H}^2$ from Corollary \ref{cont diff}. The bundle projection is a submersion, so by Proposition \ref{mo buttah} this composition has full rank derivative at every point of the open dense subset $U\subset T_k$ identified there. As in the previous proof, the implicit function theorem now implies that for every $\bs\in U$, there is an open neighborhood of $\psi_k(\bs)$ contained in $\psi_k(U)\subset\psi_k(T_k)$. Thus $\psi_k(U)$ is contained in the interior of the set of walk endpoints. Moreover, since $U$ is dense in $T_k$, its image is dense in the set of endpoints.
\end{proof}

\begin{remark}\label{diameter bound} The \mbox{\rm diameter} of the set of endpoints of walks with jumps having jump angle $\pi/2$ and other data identical to that of Corollary \ref{rich behavior}, meaning the supremum of the pairwise distances between these endpoints, can be explicitly bounded below along the lines of Proposition \ref{distinct endpoints}. A computation parallel to the proof of that result shows that the distance $d$ between the endpoints of the walks with jumps corresponding to the vertices $(0,\hdots,0)$ and $(T,\hdots,T)$ of $T_k$ satisfies:
\[ \cosh(d)-1 = (\cosh(k\ell)-1)(\cosh(sT)-1)(\cosh(k\ell)\cosh(sT)-1).\]
The diameter of the set of endpoints is thus bounded below by $d$, which can be seen to increase linearly with any of the individual quantities featured in its definition.

This exhibits a stark qualitative difference between the Euclidean and hyperbolic contexts: if one replaced ``$\mathbb{H}^2$'' by ``$\mathbb{R}^2$'' in its statement then the resulting set of endpoints would have diameter $0$, being a single point.
\end{remark}

\begin{proposition}\label{P: different walks same endpoints}
Given the data $o$, $\textbf{v}$, $s$, $\theta$, $\ell$, $T$, there exist two distinct sequences of jump times $(t_1,\ldots,t_n)$ and $(t_1',\ldots,t_n')$ that the corresponding walks with jumps have the same endpoints. In particular, the associated isometries $w = g_{t_1}fg_{t_2-t_1}f\cdots g_{t_{n}-t_{n-1}}fg_{T-t_{n}}$ and $w' = g_{t_1'}fg_{t_2'-t_1'}f\cdots g_{t_{n}'-t_{n-1}'}fg_{T-t_{n}'}$ are equal.
\end{proposition}

\begin{proof}
By Proposition~\ref{P: map to bundle not inj}, the map $\psi$ is not injective. Then $(w(o),w_*(\textbf{v})) = (w'(o),w_*'(\textbf{v}))$. By Proposition~\ref{P: unit tangent bundle determines isometry}, $w = w'$. 
\end{proof}

\begin{definition}
A \textbf{semigroup} is a set together with an associative operation. 
A \textbf{free semigroup} is one that is isomorphic to the semigroup generated by a set with no relations. 
\end{definition}

\begin{cor}\label{C: no free semigroup}
Given the data, $o$, $\textbf{v}$, $s$, $\theta$, $\ell$, $T$ the isometry $f$ and the one parameter family $g_t$ do not generate a free semigroup. 
\end{cor}

\input{IntermediateScaleSameEndpts}

\section{Consequences of negative curvature leading to distinct endpoints}\label{S: difft endpts}

In this section, a counterpoint to the previous one, we show that both the pattern and number of jumps can affect the endpoint of a walk with jumps. First, in Section \ref{subsec: no return}, we give a criterion for producing walks with identical initial data, duration, and number of jumps but different endpoints. We subsequently prove quasigeodesicity of walk-with-jumps paths, in Section \ref{subsec: qg}, a result that has as a particular consequence that having a different enough number of jumps leads to different endpoints.

\subsection{Future quadrants and points of no return}\label{subsec: no return}

\begin{lemma} \label{L: separation} For a walk with jumps with the data of Lemma \ref{wwj via isoms}, the distinct walk axes of Definition \ref{D: walk and jump axes} are 
pairwise disjoint. 
Furthermore if 
$j<i<k$, then $\gamma_j$ and $\gamma_k$ lie in distinct components of $\Htwo\setminus\gamma_i$. For each $i$, 
the shortest distance $d$ from $\gamma_i$ to $\gamma_{i+1}$ satisfies $\sinh(d/2) = \sin\theta\sinh(\ell/2)$. 

The distinct jump axes are also pairwise disjoint  
and satisfy the analogous separation property; and for $i>1$,  
the shortest distance $x_i$ from $\lambda_{i-1}$ to $\lambda_i$ satisfies $\sinh(x_i/2) = \sin\theta \sinh(s(t_i-t_{i-1})/2)$.\end{lemma}

\begin{figure}[ht]
\begin{tikzpicture}

\draw (-3,0) -- (5,0);
\draw (-3,2) -- (5,2);
\draw (-1,-1) -- (3,3);
\draw [very thick, -stealth] (-2,0) -- (0,0) -- (2,2) -- (4,2);

\node [above] at (-2,0) {$\gamma_{i-1}$};
\node [above] at (-2,2) {$\gamma_{i}$};
\node [right] at (-.75,-.75) {$\lambda_{i}$};
\node [above right] at (0.25,0) {$\theta$};
\node [above right] at (2.25,2) {$\theta$};

\draw [thick, dashed] (1,2) -- (1,0);
\node [above left] at (0.5,0.5) {$\ell/2$};
\node [right] at (1,0.5) {$d/2$};

\draw (1,1.8) -- (1.2,1.8) -- (1.2,2);
\draw (0.8,0) -- (0.8,0.2) -- (1,0.2);

\end{tikzpicture}
\caption{Two walk axes $\gamma_{i-1}$ and $\gamma_i$, and a jump axis $\lambda_i$, with a segment of a walk-with-jumps path in bold. The shortest distance between walk axes is along the dashed arc.}\label{crossing axes}
\end{figure}

\begin{proof} We can take the $\gamma_i$ and $\lambda_j$ to be oriented ``in the direction of travel'' (i.e.,~in the direction that the relevant conjugate of the $g_t$ or $f$ translates). For a fixed $i>0$, $\gamma_{i-1}$ and $\gamma_i$ both intersect $\lambda_i$. At each intersection, the angle from its positive direction to that of $\lambda$ is equal to $\theta$. If $\gamma_i$ and $\gamma_{i-1}$ had a point of intersection they would thus form a triangle, together with $\lambda_i$, whose angles at the edge contained in $\lambda_i$ were $\theta$ and $\pi-\theta$. This would contradict that the angle sum of a hyperbolic triangle with at least one compact vertex is less than $\pi$.

The same argument also rules out $\gamma_i$ and $\gamma_{i-1}$ sharing an ideal endpoint, so they are at a nonzero distance from each other. The minimum distance $d$ is attained at endpoints of a unique geodesic arc that meets each of $\gamma_{i-1}$ and $\gamma_i$ at right angles.

To determine $d$, we first note that the $\pi$-rotation around the midpoint of the arc of $\lambda_i$ that joins $\gamma_{i-1}$ to $\gamma_i$ takes $\lambda_i$ to itself and exchanges $\gamma_{i-1}$ with $\gamma_i$. (This can be seen from the fact that it interchanges their tangent vectors at their points of intersection with $\lambda_i$, which determine the geodesics.) The geodesic arc joining $\gamma_{i-1}$ to $\gamma_i$ is therefore also preserved by this rotation, so it contains its fixed point as a bisector.

The formula for $d$ now follows from the hyperbolic law of sines, see Figure \ref{crossing axes}. The distance $x_i$ between $\lambda_{i-1}$ and $\lambda_i$ is established analogously, since these geodesics are joined by an arc of $\gamma_i$ of length $s(t_i-t_{i-1})$. (The main difference from the previous case being that this depends on $i$, since the time intervals between jumps can vary)

We finally note that by construction of the walk with jumps, the segment of $\lambda_{i}$ joining $\gamma_{i-1}$ to $\gamma_i$ is on the opposite side of $\gamma_i$ from the arc of $\lambda_{i+1}$ joining $\gamma_i$ to $\gamma_{i+1}$ (for $0<i<n$). From this, it follows that $\gamma_i$ separates $\gamma_{i-1}$ from $\gamma_{i+1}$. A quick inductive argument now ensures for all $j<i<k$ that $\gamma_j$ and $\gamma_k$ lie in opposite complementary components of $\gamma_i$.
\end{proof}

\begin{cor}\label{angle bound} For a walk with jumps with the data of Lemma \ref{wwj via isoms}, the parametrization of the associated walk-with-jumps path given in Definition \ref{D: walkswithjumpspath} defines an embedding of the parameter interval to $\mathbb{H}^2$. 

For each $i\ge 1$, any points $a$, $p_i$, and $b$ on the walk-with-jumps path such that the parameter value of $a$ (respectively, $b$) is less (resp.~greater) than that of $p_i$ determine a triangle in $\mathbb{H}^2$ whose interior angle at $p_i$ is at least $\pi-\theta$. The same angle bound likewise holds for an analogous triangle with vertices at $a$, $q_i$, and $b$.\end{cor}

\begin{proof} Since the parametrization given in Definition \ref{D: walkswithjumpspath} is by arclength on any sub-interval of the parameter interval that maps to a walk or jump segment, it is injective on such sub-intervals. Moreover, walk and jump segments that share an endpoint do not intersect outside this endpoint, since it is a general property of hyperbolic geodesics that they intersect in at most a single point. Therefore for parameter values $t$ and $t'\ne t$ mapping to the same point, we may assume that they do not belong to the same or adjacent sub-intervals as above. However then the walk or jump segments that they map to, which are not identical and do not share an endpoint, are separated by the walk or jump axis containing any segment that lies between these two on the walk-with-jumps path, by Lemma \ref{L: separation}.

We address the triangle $\triangle$ with vertices at $a$, $p_i$, and $b$; the other case is similar. By Lemma \ref{wwj via isoms}, $o$ lies on $\gamma_1$, $p_i$ on $\gamma_i\cap\lambda_i$, and $q_i$ on $\lambda_i\cap\gamma_{i+1}$. Thus by Lemma \ref{L: separation}, the edge of $\triangle$ joining $a$ to $p_i$ lies entirely in the half-plane bounded by $\gamma_{i}$ that does not contain $q_i$ and the half-plane bounded by $\lambda_i$ that does not contain $p_{i+1}$. These half-planes have angle of intersection $\theta$ at $p_i$, so the interior angle of $T$ here exceeds the complementary angle $\pi-\theta$, see Figure~\ref{F: angle bound}.\end{proof}

\begin{figure}[h]
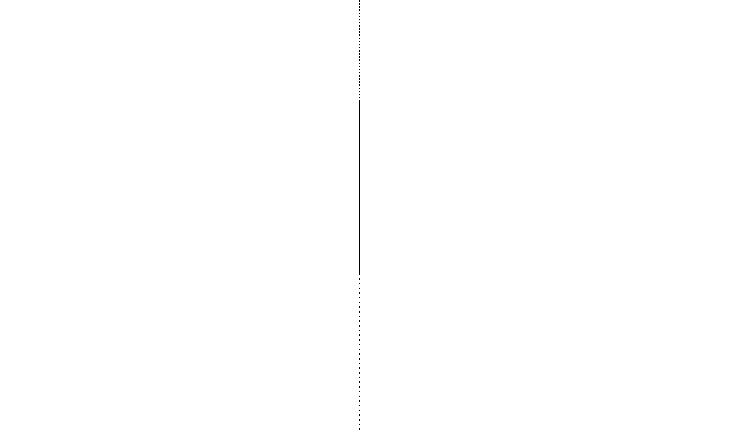
\caption{To accompany Corollary~\ref{angle bound}}
\label{F: angle bound}
\end{figure}

We begin by formally defining a notion that has appeared already in the proof of Lemma \ref{wwj via isoms}. 

\begin{definition}\label{D: initial segment} Suppose a walk with jumps is specified by initial data $o$, $\mathbf{v}$, $s$, $\theta$, $\ell$; duration $T$; jump time sequence $(t_1,\hdots,t_n)$; and burst vector $(j_1,\hdots,j_n)$. For any $T_0\le T$, the \textit{duration-$T_0$ initial segment} of the given walk with jumps is the one sharing the same initial data; with duration $T_0$; sequence of jump times $(t_1,\hdots,t_k)$, where $k\leq n$ is the largest index such that $t_k\leq T_0$; and burst vector $(j_1,\hdots,j_k)$ for the same $k$. Conversely, we say that the original walk with jumps is obtained by \textit{prolonging} any initial segment.
\end{definition}

\begin{remark} Let $w\in\PSLR$ be the word in syllable form with respect to some $f,g_t\in \PSLR$ associated to a walk with jumps by Lemma \ref{wwj via isoms}. For any initial segment of the given walk with jumps, the syllable form of its associated word $w_0\in\PSLR$ is a \mbox{\rm prefix} of $w$. Specifically, $w=w_0w_1$ for some $w_1$, where the syllable form of $w$ can be determined from the syllable forms of $w_0$ and $w_1$ by replacing the last syllable $g_{s_0}$ of $w_0$ and the first syllable $g_{s_1}$ of $w_1$ with a single syllable $g_{s_0+s_1}$.\end{remark}

The result below is another consequence of Lemma \ref{L: separation}. Given a bi-infinite geodesic $\gamma\subset\mathbb{H}^2$ such as one of the walk or jump axes of a walk with jumps, a \textit{half-plane bounded by $\gamma$} is the closure of one component of $\mathbb{H}^2-\gamma$. Its \textit{complementary} half-plane is the closure of the other component.

\begin{cor}\label{quadrants} Suppose a walk with jumps is specified by initial data $o$, $\mathbf{v}$, $s$, $\theta$, $\ell$; duration $T$; jump time sequence $(t_1,\hdots,t_n)$, $n\ge 1$; and burst vector $(j_1,\hdots,j_n)$. Let $\gamma_g$ and $\lambda_f$ be as in Lemma \ref{wwj via isoms first half}, and let $\gamma_i$, for $i\in\{1,\hdots,n+1\}$, and $\lambda_j$, for $j\in\{1,\hdots,n\}$, be the walk and jump axes from Definition \ref{D: walk and jump axes}. Then:\begin{itemize}
    \item Let $H_0^+$ be the half-plane bounded by $\gamma_g = \gamma_1$ that contains $\gamma_2$ and $K_0^+$ the half-plane bounded by $\lambda_f$ that contains $\lambda_1$. Then $H_0^+\cap K_0^+$ contains the entire associated walks-with-jumps path.
    \item For each $i\in\{1,\hdots,n\}$ let $H_i^-$ be the half-plane bounded by $\gamma_{i+1}$ that contains $o$ and $K_i^-$ the half-plane bounded by $\lambda_i$ that contains $o$. The entire walks-with-jumps path of the duration-$t_i$ initial segment is contained in $H_i^-\cap K_i^-$, and the remainder of the full walk-with-jumps path is contained in the intersection of the complementary half-planes $H_i^+\cap K_i^+$.
\end{itemize}
Moreover, taking $\lambda_{n+1}$ to be the geodesic through the endpoint of the walk with jumps at an angle of $\theta$ to $\gamma_{n+1}$, measured counterclockwise from the endvector, there is a half-plane $K_{n+1}^-$ bounded by $\lambda_{n+1}$ such that $H_n^-\cap K_{n+1}^-$ contains the entire walk-with-jumps path.
\end{cor}

\begin{proof} We claim that $H_0^+$ contains each axis $\gamma_i$ for $i\ge 1$. Of course $H_0^+$ contains its boundary $\gamma_1$, and the case $i=2$ holds by definition. Now for $i\ge 2$, supposing that $\gamma_j\subset H_0^+$ for each $j\le i$, we claim that $\gamma_{i+1}$ is also contained in $H_0^+$. This follows from Lemma \ref{L: separation}, which asserts that $\gamma_i$ separates $\gamma_{i+1}$ from $\gamma_{i-1}$, so since $\gamma_{i-1}$ and $\gamma_i$ lie in $H_0^+$, so does $\gamma_{i+1}$. The claim thus follows by induction. 

By Lemma \ref{wwj via isoms}, it now follows that $H_0^+$ contains each walk segment of the associated walk-with-jumps-path. And since each jump segment is a geodesic arc joining endpoints of two walk segments, and $H_0^+$ is convex, it also contains each jump segment. Therefore $H_0^+$ contains the entire associated walk-with-jumps path. An entirely analogous argument shows that $K_0^+$ does as well, and the first bulleted claim is proved.

We now prove the second claim.  We claim that if $j\le i+1$, then $\gamma_j$ is in $H_i^-$, the half plane bounded by $\gamma_{i+1}$ that contains $o$. 
Indeed, if $j=i+1$ the result is clear, so we assume $j<i+1$. In the special case, $j=0$, Lemma~\ref{L: separation} implies $\gamma_0$ is contained in a distinct half plane bounded by $\gamma_0$ and $\gamma_0$ contains $o$. 
For $0<j<i+1$, Lemma~\ref{L: separation} implies that $\gamma_j$ separates $o\in \gamma_0$ from $\gamma_j$. Therefore, the half plane bounded by $\gamma_{i+1}$ that contains $\gamma_j$ must be the one that contains $o$. 

If $x$ lies in the duration--$t_i$ in initial segment, then $x$ either lies on one of the $\gamma_j$ for $j<i+1$ or lies on a geodesic between $\gamma_{j_1},\gamma_{j_2}$ where $j_1,j_2\le i+1$. In the first case, we immediately have that $x$ lies in $H_i^-$, and in the second $x$ lies in $H_i^-$ by convexity of half planes. 
A similar argument shows that $x$ also lies in $K_i^-$. Then $x\in H_i^-\cap K_i^-$. 

On the other hand, if $y$ lies in the remainder of the full walk-with-jumps path, then $y$ lies on some $\gamma_j$ for $j\ge i+1$ or $y$ lies on a geodesic between $\gamma_{j_1},\gamma_{j_2}$ where $j_1,j_2\ge i+1$. Lemma~\ref{L: separation} implies that if $j> i+1$, then $\gamma_{i+1}$ separates $\gamma_j$ from $\gamma_0$ which contains $o$. Therefore, for $j\ge i+1$, $\gamma_j$ lies in $H_i^+$, the half plane bounded by $\gamma_{i+1}$ that does not contain $o$. Therefore, $y\in H_i^+$ by convexity of $H_i^+$. A similar argument shows $y$ also lies in $K_i^+$. Therefore, $y\in H_i^+\cap K_i^+$. 

We now address the Corollary's final claim. Note that each of $\lambda_{n+1}$ and $\lambda_n$ intersect $\gamma_{n+1}$ at an angle of $\theta$: the former by construction, and the latter as observed in Lemma \ref{L: separation}. They coincide if $t_n = T$; otherwise they do not intersect (also as observed in Lemma \ref{L: separation}). In the former case, take $K_{n+1}^- = K_n^-$. In the latter, let $K_{n+1}^-$ be the half-plane bounded by $\lambda_{n+1}$ and containing $\lambda_n$. Then $K_{n+1}^-$ contains the $(n+1)$st walk segment, since the endpoints of this segment lie on $\lambda_{n+1}$ and $\lambda_n$, and it contains the half-plane $K_n^-$ opposite this segment. Therefore by the previous claim, $K_{n+1}^-$ contains the entire walk-with-jumps path.
\end{proof}

\begin{definition}\label{future quad}
Motivated by the second bullet of Corollary \ref{quadrants}, for $H_n^+$, $\lambda_{n+1}$, and $K_{n+1}^-$ as defined there, we call $K_{n+1}^+$ the half-plane bounded by $\lambda_{n+1}$ opposite $K_{n+1}^-$ and deem the intersection $H_n^+\cap K_{n+1}^+$ of half-planes the \textit{future quadrant} of the walk with jumps specified there.\end{definition}

\subsection{Quasigeodesicity}\label{subsec: qg}
In this section we will show that the walk-with-jumps paths in $\mathbb{H}^2$ from Definition \ref{D: walkswithjumpspath} are ``quasi-geodesic''. 
This property is well known to have important consequences in the study of negatively curved spaces. A key consequence for our application, recorded in Corollary \ref{C: jump simplices} below, is that walks with different (enough) \textit{numbers} of jumps have different endpoints, irrespective of any other consideration. As in the previous subsection, our results here are \textit{effective}, meaning that we produce explicit constants.

\begin{definition}\label{D: qg}
For a fixed $\lambda\in(0,1]$ and $\epsilon \ge 0$, a \textit{$(\lambda,\epsilon)$-quasi-geodesic} in a metric space $(X,d)$ is a map $c: I\to X$, where $I\subset\mathbb{R}$ is an interval, such that for all $t,t'\in I$:
\[ \lambda|t-t'| - \epsilon \le d(c(t),c(t')) \le \frac{1}{\lambda} |t-t'| + \epsilon. \]
\end{definition}

This definition asserts that the path $c$ does not shrink or expand distance by more than a multiplicative factor of $\lambda$ and an additive factor of $\epsilon$. In applying the notion to a walk-with-jumps path we will leverage the fact that it is a broken geodesic. Our parametrization described in Definition \ref{D: walkswithjumpspath} is by arclength on each piece, so the right-hand inequality above will automatically hold for any $\lambda\in(0,1]$ and $\epsilon\ge 0$.

\begin{proposition}[Hyperbolic Law of Cosines, {\cite[I.2.7]{BridsonHaefliger}}]\label{HLC}
Let $\triangle$ be a hyperbolic triangle with vertices $A,B,C$. Let $a = d(B,C)$, $b=d(C,A),$ and $c=d(A,B)$. Let $\theta$ denote the vertex angle at $C$. Then
\[\cosh c =\cosh a\cosh b - \sinh a\sinh b \cos\theta\]
In particular, if $|\theta|= \frac\pi2$, then:
\[\cosh c = \cosh a\cosh b\]
\end{proposition}

\begin{lemma}\label{angle theta defect} For a triangle $\triangle\subset\mathbb{H}^2$ with vertices $A$, $B$, and $C$ and opposite side lengths $a$, $b$, and $c$ respectively, if $\triangle$ has interior angle $\theta$ at $C$ then
\[ a + b - c < \delta_{\theta} \doteq \ln\left( \frac{2}{1-\cos\theta}\right). \]
This bound is sharp, not attained, but asymptotically approached as $a = b\to\infty$.
\end{lemma}

\begin{remark} We fix $\theta = \pi/2$ to illustrate the strong contrast between $\mathbb{R}^2$ and $\mathbb{H}^2$. By the result above, for a right triangle in $\mathbb{H}^2$ with hypotenuse of length $c$, the difference $a+b-c$ is universally bounded above by $\ln 2$. On the other hand, for an isosceles right triangle in $\mathbb{R}^2$ with shorter side length $a=b$ the difference $a+b-c = (2-\sqrt{2})a$ increases linearly with $a$.\end{remark}

\begin{proof} Using the hyperbolic law of cosines from Proposition \ref{HLC}, we consider $a+b - c$ as a function $f_{\theta}(a,b)$ on the quadrant $\{a,b\geq 0\}$ in the $ab$-plane by substituting
\[ c = \cosh^{-1}(\cosh a\cosh b - \sinh a\sinh b \cos\theta). \]
A somewhat messy calculus computation gives that the gradient vector $\nabla f_{\theta}(a,b)$ is a positive scalar multiple of $(\sinh^2 B,\sinh^2 A)$. Thus at any $(a,b) \ne(0,0)$ (the global minimum of $f_{\theta}$), the gradient points away from the origin and toward the diagonal $(x,x)$. So for instance, for any fixed $X>0$, the maximum of $f_{\theta}$ on $[0,X]^2$ occurs at $(X,X)$. Moreover, $f_{\theta}(x,x)$ is an increasing function of $x$, and it follows that the values of $f_{\theta}$ on the entire quadrant are bounded above by $\delta_{\theta} = \lim_{x\to\infty} f_{\theta}(x,x)$.

To evaluate the limit, we exponentiate $f_{\theta}(x,x)$ and use the fact that $\cosh^{-1}y = \ln\left(y+\sqrt{y^2-1}\right)$ to write 
\[ e^{f_{\theta}(x,x)} = \frac{e^{2x}}{\cosh^2x - \sinh^2x\cos\theta + \sqrt{(\cosh^2x - \sinh^2x\cos\theta)^2 - 1}} \]
Multiplying top and bottom by $e^{-2x}$ then evaluating the limit and taking a natural log yields the formula for $\delta_{\theta}$ given above.
\end{proof}

\begin{proposition}\label{trig qg}
For a walk with jumps with initial data $o$, $\mathbf{v}$, $s$, $\ell$, $\theta$; of duration $T$ and with jump times $(t_1,\hdots,t_n)$ and burst vector $(j_1,\hdots,j_n)$, the distance $d$ from $o$ to the walk's endpoint satisfies
\[ sT + \ell N \ge d \ge sT + \ell N - 2n\delta_{\pi-\theta}, \]
for $\delta_{\theta}$ as in Lemma \ref{angle theta defect}, where $N = \sum_{i=1}^n j_i$ is the number of jumps. Moreover, supposing that $\delta_{\pi-\theta} < \ell/2$, define
\begin{equation} \lambda = \frac{\ell - 2\delta_{\pi-\theta}}{\ell}. \label{Eq: jump count defect}\end{equation}
Then the associated walk-with-jumps path is a $(\lambda,2\delta_{\pi-\theta})$-quasi-geodesic.
\end{proposition}

\begin{remark} The length lower bound above implies that ``more bursty'' walks are closer to geodesic. To support this assertion, note that $n\leq N$ since $N$ is the total number of jumps whereas $n$ is the number of jump {\rm times}. Thus it is reasonable to take $N-n$ to measure the ``burstiness'' of the walk: for a fixed $N$, the growth of $N-n$ corresponds to the same number of jumps being concentrated at fewer times. As Proposition \ref{trig qg}'s lower bound decreases with $n$, it also increases with $N-n$.\end{remark}

\begin{proof} The upper bound on the distance between endpoints follows from the fact that the walk-with-jumps path of Definition \ref{D: walkswithjumpspath} is a broken geodesic joining $o$ to the walk's endpoint, of total length $sT + \ell\cdot\sum_i j_i$. This also implies that the distance between any two points on the path is bounded above by their distance \textit{along} the path, ie.~the difference between their parameter values if it is parametrized piecewise by arclength as in the Definition. This implies the upper bound needed for quasigeodesicity in Definition \ref{D: qg}.

To prove the Proposition's lower bound on the distance between endpoints we sequentially apply Lemma \ref{angle theta defect} and Corollary \ref{angle bound}, first to the triangle with vertices $o$, $p_i$, and $q_i$ then to the triangle with vertices $o$, $q_i$, and $p_{i+1}$, for $i$ increasing from $i$ to $n$. (For $i=n$ we interpret ``$p_{i+1}$'' to refer to the walk's endpoint.) From these we obtain inductively that for each $i$, the distance from $o$ to $q_i$ is at least $st_i + \ell\cdot\sum_{k=1}^i j_k - (2i-1)\delta_{\pi-\theta}$, and from $o$ to $p_{i+1}$, at least $st_{i+1}+ \ell\cdot\sum_{k=1}^i j_k - 2i\delta_{\pi-\theta}$
(where for $i = n$ we interpret ``$t_{i+1}$'' as $T$).

For arbitrary $x$ and $y$ on the walk-with-jumps path, an entirely analogous argument shows that the distance from $x$ to $y$ is at least $d_0 - k\delta_{\pi-\theta}$, where $d_0$ is the length of the sub-path bounded by $x$ and $y$ and $k$ is the number of $p_i$ and $q_i$ that lie within this sub-path. The sub-path bounded by $x$ and $y$ contains least $\frac{k-2}{2}$ jump segments, with the lower bound attained only if $x$ and $y$ each lie in separate jump segments. Thus using that $d_0 > \frac{k-2}{2}\ell$, we have
\[ d(x,y) \ge d_0 - \frac{k-2}{2}(2\delta_{\pi-\theta}) - 2\delta_{\pi-\theta} \ge d_0\left(1 - \frac{2\delta_{\pi-\theta}}{\ell}\right) - 2\delta_{\pi-\theta}. \]
The right-hand quantity in parentheses above is exactly $\lambda$ from (\ref{Eq: jump count defect}), so we have proved the needed lower bound for quasigeodesicity.\end{proof}

\begin{cor}\label{C: jump simplices} Let $o$, $\mathbf{v}$, $s$, $\ell$, $\theta$ be the initial data of a walk with jumps. Supposing that $\delta_{\pi-\theta} < \ell/2$, for $\delta$ as in Lemma \ref{angle theta defect}, define $\lambda$ as in (\ref{Eq: jump count defect}). Then for any natural numbers $M$ and $N$ with $M<\lambda N$, and any fixed $T>0$, no two walks with jumps that share the initial data above and duration $T$ have the same endpoint if the first has $M$ jumps and the second, $N$.\end{cor} 

The Corollary follows directly by comparing Proposition \ref{trig qg}'s upper bound for the walk with $M$ jumps to the lower bound for the walk with $N$ jumps, and applying the fact that $M<\lambda N$ and $n\leq N$ to show that the endpoint of the walk with $M$ jumps is closer to $o$ than the endpoint of the walk with $N$ jumps.

As $\ell\to\infty$ in \eqref{Eq: jump count defect}, $\lambda\to 1$. Therefore, if $\ell$ is chosen to be sufficiently long (relative to $\theta$), then Corollary~\ref{C: jump simplices} implies that walks with jumps paths with different numbers of jumps have distinct endpoints.

\input{SepEndpts}

\section{Further Questions}\label{sec: FurQue}

The above results establish some of the basic features of walks with jumps, but leave open a number of questions of mathematical and biological interest.  First, there are a number of natural distances that can be defined on walk-with-jump paths, including the Hausdorff distance between the paths, the $\mathbb{H}^2$-distance between their endpoints, and the edit-length distance between their sequences of jump times  \cite{VictorPurpura1997}.  The relationship between these distances is at present unclear. Related to this, the structure  of the endpoint set for walks with jumps of a fixed duration and number of jumps is not known: it is connected, but its shape remains to be characterized.

\subsection{Stability}
This subsection describes a hyperbolic-geometric direction for future study. The observation below follows from standard facts:
\begin{proposition}\label{P: stability}
Let $\sigma_1,\sigma_2$ be two walks with jumps that satisfy the hypotheses of Proposition~\ref{trig qg} so that $\sigma_1,\sigma_2$ have the same endpoints. 
Then there exists $D = D(\ell,\delta_{\pi-\theta})$ so that the Hausdorff between $\sigma_1,\sigma_2$ is bounded above by $D$. 
\end{proposition}

Proposition~\ref{P: stability} follows immediately from Proposition~\ref{trig qg} and quasigeodesic stability (see \cite[Theorem III.H.1.7]{BridsonHaefliger}). 
At large scale, Proposition~\ref{P: stability} supports the heuristic that two walks with jumps that have the same endpoints have similar behavior. 
Here are two natural questions that we have investigated but do not have clear answers to:
\begin{enumerate}
    \item For what initial data and at what scale is the bound $D = D(\ell,\delta_{\pi-\theta})$ on Hausdorff distance between walks-with-jumps with the same endpoints effective relative the overall length of the associated walks-with-jumps paths? Compounding this problem is the fact that the value $D$ can be hard to estimate. For example, the proof of \cite[Theorem III.H.1.7]{BridsonHaefliger} shows that $D$ exists and provides an indirect method to compute $D$ but there is no explicit formula for a good estimate.  
    \item Is there a natural way to translate between differences in walk data and Hausdorff distance between two walks with jumps? Could this be used with Proposition~\ref{P: stability} to distinguish walks with jumps paths from their data?
\end{enumerate}

\subsection{Other Directions}
Biological considerations suggest other directions for study and extensions. Neural spike trains are often considered to be samples from a point process; it is therefore of interest to determine how the point process model (e.g., Poisson, Poisson with refractory period, Markov, etc.), impacts the distribution of endpoints and the extent to which walks with jumps may cross. Finally, just as the activity of a single neuron can be formalized as a point process, multineuronal activity can be formalized as a labeled point process. In analogy with the  walk-with-jump model of a single neuron's activity, multineuronal activity could then be modeled as a walk with jumps in a higher-dimensional space, possibly hyperbolic or with hyperbolic subspaces, in which each neuron's activity corresponds to a different kind of jump.

\bibliography{biblio}{}
\bibliographystyle{plain}
\end{document}

%% file: spike_model_comparison_pic2.pdf_tex
\begingroup%
  \makeatletter%
  \providecommand\color[2][]{%
    \errmessage{(Inkscape) Color is used for the text in Inkscape, but the package 'color.sty' is not loaded}%
    \renewcommand\color[2][]{}%
  }%
  \providecommand\transparent[1]{%
    \errmessage{(Inkscape) Transparency is used (non-zero) for the text in Inkscape, but the package 'transparent.sty' is not loaded}%
    \renewcommand\transparent[1]{}%
  }%
  \providecommand\rotatebox[2]{#2}%
  \newcommand*\fsize{\dimexpr\f@size pt\relax}%
  \newcommand*\lineheight[1]{\fontsize{\fsize}{#1\fsize}\selectfont}%
  \ifx\svgwidth\undefined%
    \setlength{\unitlength}{439.61597353bp}%
    \ifx\svgscale\undefined%
      \relax%
    \else%
      \setlength{\unitlength}{\unitlength * \real{\svgscale}}%
    \fi%
  \else%
    \setlength{\unitlength}{\svgwidth}%
  \fi%
  \global\let\svgwidth\undefined%
  \global\let\svgscale\undefined%
  \makeatother%
  \begin{picture}(1,0.52172768)%
    \lineheight{1}%
    \setlength\tabcolsep{0pt}%
    \put(0,0){\includegraphics[width=\unitlength,page=1]{spike_model_comparison_pic2.pdf}}%
    \put(0.18482855,0.35279943){\color[rgb]{0,0,0}\makebox(0,0)[lt]{\lineheight{1.25}\smash{\begin{tabular}[t]{l}\Huge \bf A\end{tabular}}}}%
    \put(0.51955113,0.00784745){\color[rgb]{0,0,0}\makebox(0,0)[lt]{\lineheight{1.25}\smash{\begin{tabular}[t]{l}\Huge \bf C\end{tabular}}}}%
    \put(0,0){\includegraphics[width=\unitlength,page=2]{spike_model_comparison_pic2.pdf}}%
    \put(-0.00080846,0.41156507){\color[rgb]{0,0,0}\makebox(0,0)[lt]{\lineheight{1.25}\smash{\begin{tabular}[t]{l}$0$\end{tabular}}}}%
    \put(0.0918024,0.41276251){\color[rgb]{0,0,0}\makebox(0,0)[lt]{\lineheight{1.25}\smash{\begin{tabular}[t]{l}$t_1$\end{tabular}}}}%
    \put(0.22544074,0.41389212){\color[rgb]{0,0,0}\makebox(0,0)[lt]{\lineheight{1.25}\smash{\begin{tabular}[t]{l}$t_2$\end{tabular}}}}%
    \put(0.38318986,0.41197912){\color[rgb]{0,0,0}\makebox(0,0)[lt]{\lineheight{1.25}\smash{\begin{tabular}[t]{l}$T$\end{tabular}}}}%
    \put(-0.00132616,0.06969016){\color[rgb]{0,0,0}\makebox(0,0)[lt]{\lineheight{1.25}\smash{\begin{tabular}[t]{l}$0$\end{tabular}}}}%
    \put(0.09011823,0.07151227){\color[rgb]{0,0,0}\makebox(0,0)[lt]{\lineheight{1.25}\smash{\begin{tabular}[t]{l}$t_1$\end{tabular}}}}%
    \put(0.22326428,0.07096081){\color[rgb]{0,0,0}\makebox(0,0)[lt]{\lineheight{1.25}\smash{\begin{tabular}[t]{l}$t_2$\end{tabular}}}}%
    \put(0.38124172,0.07000442){\color[rgb]{0,0,0}\makebox(0,0)[lt]{\lineheight{1.25}\smash{\begin{tabular}[t]{l}$T$\end{tabular}}}}%
    \put(0,0){\includegraphics[width=\unitlength,page=3]{spike_model_comparison_pic2.pdf}}%
    \put(0.18509133,0.00768088){\color[rgb]{0,0,0}\makebox(0,0)[lt]{\lineheight{1.25}\smash{\begin{tabular}[t]{l}\Huge \bf B\end{tabular}}}}%
  \end{picture}%
\endgroup%

%% file: spike_pic2.pdf_tex
\begingroup%
  \makeatletter%
  \providecommand\color[2][]{%
    \errmessage{(Inkscape) Color is used for the text in Inkscape, but the package 'color.sty' is not loaded}%
    \renewcommand\color[2][]{}%
  }%
  \providecommand\transparent[1]{%
    \errmessage{(Inkscape) Transparency is used (non-zero) for the text in Inkscape, but the package 'transparent.sty' is not loaded}%
    \renewcommand\transparent[1]{}%
  }%
  \providecommand\rotatebox[2]{#2}%
  \newcommand*\fsize{\dimexpr\f@size pt\relax}%
  \newcommand*\lineheight[1]{\fontsize{\fsize}{#1\fsize}\selectfont}%
  \ifx\svgwidth\undefined%
    \setlength{\unitlength}{312.53181758bp}%
    \ifx\svgscale\undefined%
      \relax%
    \else%
      \setlength{\unitlength}{\unitlength * \real{\svgscale}}%
    \fi%
  \else%
    \setlength{\unitlength}{\svgwidth}%
  \fi%
  \global\let\svgwidth\undefined%
  \global\let\svgscale\undefined%
  \makeatother%
  \begin{picture}(1,0.85287694)%
    \lineheight{1}%
    \setlength\tabcolsep{0pt}%
    \put(0,0){\includegraphics[width=\unitlength,page=1]{spike_pic2.pdf}}%
    \put(0.0309029,1.97752839){\color[rgb]{0,0,0}\makebox(0,0)[lt]{\lineheight{1.25}\smash{\begin{tabular}[t]{l}A\end{tabular}}}}%
    \put(0.0217585,1.51195005){\color[rgb]{0,0,0}\makebox(0,0)[lt]{\lineheight{1.25}\smash{\begin{tabular}[t]{l}B\end{tabular}}}}%
    \put(0.56349375,1.51025317){\color[rgb]{0,0,0}\makebox(0,0)[lt]{\lineheight{1.25}\smash{\begin{tabular}[t]{l}C\end{tabular}}}}%
    \put(0,0){\includegraphics[width=\unitlength,page=2]{spike_pic2.pdf}}%
    \put(-0.21603749,2.09712628){\color[rgb]{0,0,0}\makebox(0,0)[lt]{\lineheight{1.25}\smash{\begin{tabular}[t]{l}0\end{tabular}}}}%
    \put(-0.08499283,2.09845918){\color[rgb]{0,0,0}\makebox(0,0)[lt]{\lineheight{1.25}\smash{\begin{tabular}[t]{l}t1\end{tabular}}}}%
    \put(0.10305943,2.09818066){\color[rgb]{0,0,0}\makebox(0,0)[lt]{\lineheight{1.25}\smash{\begin{tabular}[t]{l}t2\end{tabular}}}}%
    \put(0.32672335,2.098241){\color[rgb]{0,0,0}\makebox(0,0)[lt]{\lineheight{1.25}\smash{\begin{tabular}[t]{l}T\end{tabular}}}}%
    \put(-0.21422038,1.62366565){\color[rgb]{0,0,0}\makebox(0,0)[lt]{\lineheight{1.25}\smash{\begin{tabular}[t]{l}0\end{tabular}}}}%
    \put(-0.08134993,1.62622861){\color[rgb]{0,0,0}\makebox(0,0)[lt]{\lineheight{1.25}\smash{\begin{tabular}[t]{l}t1\end{tabular}}}}%
    \put(0.10551278,1.62884672){\color[rgb]{0,0,0}\makebox(0,0)[lt]{\lineheight{1.25}\smash{\begin{tabular}[t]{l}t2\end{tabular}}}}%
    \put(0.32475881,1.62410769){\color[rgb]{0,0,0}\makebox(0,0)[lt]{\lineheight{1.25}\smash{\begin{tabular}[t]{l}T\end{tabular}}}}%
    \put(0,0){\includegraphics[width=\unitlength,page=3]{spike_pic2.pdf}}%
    \put(0.50688421,0.38516472){\color[rgb]{0,0,0}\makebox(0,0)[lt]{\lineheight{1.25}\smash{\begin{tabular}[t]{l}$sT$\end{tabular}}}}%
    \put(0.63415043,0.45473715){\color[rgb]{0,0,0}\makebox(0,0)[lt]{\lineheight{1.25}\smash{\begin{tabular}[t]{l}spike at time $T$\end{tabular}}}}%
    \put(0.21502026,0.47340284){\color[rgb]{0,0,0}\makebox(0,0)[lt]{\lineheight{1.25}\smash{\begin{tabular}[t]{l}spike at time $0$\end{tabular}}}}%
    \put(0.48482468,0.57860942){\color[rgb]{0,0,0}\makebox(0,0)[lt]{\lineheight{1.25}\smash{\begin{tabular}[t]{l}$sT$\end{tabular}}}}%
  \end{picture}%
\endgroup%

%% file: NewProofsForSection2.tex
\subsection{Background and set-up}\label{S: first obs}

Standard facts about the geometry of $\mathbb{H}^2$ imply the following:
\begin{lemma}\label{wwj via isoms first half} The initial data $o$, $\bv$, $s$, $\theta$, and $\ell$ 
of a walk with jumps, as in Definition \ref{D: walkswithjumps}, determine hyperbolic isometries $f\in\PSLR$ and the one-parameter family $(g_t)_{t\in\mathbb{R}}\subset\PSLR$ with the following properties:\begin{itemize}
    \item The $g_t$ share an axis $\gamma_g$ through $o$ in the direction of $\mathbf{v}$, and for each $t$, $g_t$ translates $o$ a distance of $st$ along $\gamma_g$.
    \item The axis $\lambda_f$ of $f$ contains $o$, and $f$ translates the origin $o$ a distance of $\ell$ along $\lambda_f$ at an angle of $\theta$ counterclockwise from $\gamma_g$.
\end{itemize}
\end{lemma}

\begin{remark}Some remarks about Lemma~\ref{wwj via isoms first half}:
\begin{itemize}
    \item The initial data of the walk with jumps is recovered from $f$ and the $g_t$ as follows: $o = \lambda_f\cap \gamma_g$, $\left.\frac{d}{dt}\right\vert_{t=0}g_t(o) = s\mathbf{v}$, $\ell = d(o,f(o))$, and $\theta$ is the counterclockwise angle from $\mathbf{v}$ to the direction vector from $o$ to $f(o)$.
    \item By replacing $f$ and $(g_t)$ with isometries of $\mathbb{H}^2$ that do not satisfy all requirements of Lemma~\ref{wwj via isoms first half}, we can produce ``walks with jumps'' satisfying different specifications --- for instance with travel along horocyclic arcs or geodesic equidistants.
\end{itemize}
\end{remark}

\begin{definition}\label{D: syllable length}
Let $w$ be a word on $\{f,g_t:\,t\in\reals^{\ge 0}\}$. The \textbf{syllable length} of $w$ is the smallest number $m\ge 1$ so that $w$ can be grouped as 
\begin{equation}\label{E: syllables}
w = g_{s_1}f^{n_2}g_{s_3} f^{n_4}\ldots f^{n_{m-1}}g_{s_m}.
\end{equation}
where $n_i\in \Z^{>0}$ and $s_i> 0$ for $1<i<m$. 
We call each of the $f^{n_i},g_{s_i}$ \textbf{syllables} of $w$.
We say $w$ is written in \textbf{syllabic form} if expressed as in \eqref{E: syllables}.
\end{definition}
We will henceforth abuse notation and use a word $w$ in syllabic form to refer to both the word itself and the isometry that it represents.
 
\begin{remark}
The syllabic form of a word should always have an odd number of syllables where the first and last syllables are $g_{s_0}$ and $g_{s_m}$ (which may be the identity). The intermediate syllables cannot be the identity. 
\end{remark}

\begin{example}
The syllable length of $w_1 = g_6 f^2 g_2g_3 f^3f$ is $5$ because we rewrite $w_1$ in syllabic form as $g_6f^2g_{5} f^4g_0$. 
The first syllable of the syllabic form is $g_6$, the second syllable is $f^2$ and the third syllable is $g_5$. 
\end{example}

\begin{definition}\label{D: walk and jump axes}
For a walk with jumps specified by $o$, $\bv$, $s$, $\theta$, $\ell$, $T$, $(t_1,\hdots,t_n)$,  $(j_1,\hdots,j_n)$, taking $f$ and $\{g_t\}_{t\in\mathbb{R}}$ as in Lemma \ref{wwj via isoms first half}, let 
\[ w =  g_{t_1}f^{j_1}g_{t_2-t_1}f^{j_2}\cdots g_{t_{n}-t_{n-1}}f^{j_n}g_{T-t_{n}}. \]
Let $w_0$ be the identity, $w_1 = g_{t_1}$ and for all $1\le i\le |w|$, let $w_i$ be the product of the first $i$ syllables of $w$. The \textit{walk axes} are the geodesics  $\gamma_i \doteq w_{2(i-1)}(\gamma_g)$  for $i\in\{1,\hdots,n+1\}$, and the \textit{jump axes} are the geodesics $\lambda_i \doteq w_{2i-1}(\lambda_f)$ for $i\in\{1,\hdots,n\}$, where $\gamma_g$, $\lambda_f$ are as in Lemma \ref{wwj via isoms first half}. 
\end{definition}

\begin{lemma}\label{wwj via isoms}
For the walk with jumps specified by initial data $o,\,\bv,\,s,\,\theta,\,\ell,$ and jump data $T$, $(t_1,\hdots,t_n)$, and $(j_1,j_2,\hdots,j_n)$, 
the endpoint-endvector pair is $(w(o),w_*(\bv))$, where $w$ from Definition \ref{D: walk and jump axes} is regarded as an isometry of $\mathbb{H}^2$, and $w_*$ is its derivative at $o$. 
Moreover, for each $i\in\{1,\hdots,n+1\}$, $\gamma_i$ from Definition \ref{D: walk and jump axes} contains the $i$th walk segment of the associated walk-with-jumps path, as described in Definition \ref{D: walkswithjumpspath}, and $\lambda_i$ contains the $i$th jump segment for $i\leq n$. 
\end{lemma}

\begin{proof}
The proof proceeds by induction on $n$. In all cases let $\bu$ be the unit vector at $o$ tangent to $\lambda_f$ pointing in the direction of $f(o)$, so that the counterclockwise angle from $\bv$ to $\bu$ is $\theta$.


Assume $n=1$, and suppose first that $t_1 = T$. 
Then the associated walks-with-jumps path consists of two segments: the first, of length $st_1$, contained in $\gamma_g$---which by definition equals $\gamma_1$---with one endpoint at $o$ and $\bv$ pointing into it; the second, of length $j_1\ell$, sharing the other endpoint $p_1$ of the first and at a counterclockwise angle of $\theta$ from its outward-pointing tangent vector. The derivative of $g_{t_1}$ at $o$ takes $\bv$ to this outward-pointing vector, since $g_{t_1}$ translates a distance of $st_1$ along $\gamma_g$, so it also takes $\bu$ to the vector pointing into the second segment at $p$. Therefore the second segment is contained in $g_{t_1}(\lambda_f) = \lambda_1$.

In this case $w = g_{t_1}f^{j_1}g_0$ in syllabic form (where $g_0$ is the identity). Note that $f^{j_1}(o)$ lies on $\lambda_f$ at a distance of $j_1\ell$ from $o$ in the direction of $\bu$. Therefore $g_{t_1}$ carries the segment of $\lambda_f$ bounded by $o$ and $f^{j_1}(o)$ to the second segment of the walk-with-jumps path, so $g_{t_1}f^{j_1}(o)$ is the path's endpoint. Similarly, the derivative of $f^{j_1}$ at $o$ carries $\bv$ to a vector at $f^{j_1}(o)$ at a clockwise angle of $\theta$ from the outward-pointing unit vector to the segment of $\lambda_f$ bounded by $o$ and $f^{j_1}(o)$; it then follows from the chain rule that $w_*(\bv)$ is the walk's endvector.

Still taking $n=1$, now suppose that $t_1 < T$. Then the associated walks-with-jumps path has three segments: the first two as in the previous case; and the third, of length $s(T-t_1)$, sharing the endpoint $q_1\ne p_1$ of the second segment and at a \textit{clockwise} angle of $\theta$ from its outward-pointing tangent vector. Thus by the previous case, the inward-pointing tangent vector to the third segment is $d(g_{t_1}f^{j_1})_o(\bv)$, and it follows that this segment lies in $\gamma_2\doteq g_{t_1}f^{j_1}(\gamma_g)$. Moreover since $g_{T-t_1}(o)$ is the point on $\gamma_g$ at distance $s(T-t_1)$ from $o$ in the direction of $\bv$, $(g_{t_1}f^{j_1})(g_{T-t_1}(o))$ is the far endpoint of the third walk segment from $q_1$.

Since $w = g_{t_1}f^{j_1}g_{T-t_1}$ in this case, the above states exactly that $w(o)$ is the endpoint of the walk with jumps. And since $d(g_{T-t_1})_o(\bv)$ is the outward-pointing tangent vector at $g_{T-t_1}(o)$ to the sub-arc of $\gamma_g$ bounded by $o$ and $g_{T-t_1}(o)$, it follows from the chain rule as in the previous subcase that $w_*(\bv)$ is the endvector of the walk with jumps. This proves the Lemma's $n=1$ case.

The $n>1$ case is analogous, after applying the inductive hypothesis to the sub-walk with jumps having the same initial data, duration $t_n$, jump times $(t_1,\hdots,t_{n-1})$, and burst vector $(j_1,\hdots,j_{n-1})$. The product $w_{2n-1}$ of the first $2n-1$ syllables of $w$ then plays the role of $g_{t_1}$ in the $n=1$ case above, as by the inductive hypothesis, $(w_{2n-1}(o),(w_{2n-1})_*(\bv))$ is the endpoint-endvector pair of the sub-walk with jumps. That is, $w_{2n-1}(o)$ is the far  endpoint from $o$ of the walk-with-jumps path associated to the sub-walk---ie.~the union of the first $2n-1$ segments of the path associated to the walk specified in the Lemma---and $(w_{2n-1})_*(\bv)$ is the outward-pointing tangent vector to the $(2n-1)^{\mathrm{st}}$ segment. The two sub-cases now follow as in the the $n=1$ case, replacing $j_1$ there with $j_n$ here and $T-t_1$ with $T-t_n$.
\end{proof}

%% file: IntermediateScaleSameEndpts.tex

\section{More different walks with identical endpoints}\label{subsec:more diff}

The previous section's results, notably Corollary~\ref{C: different walks same endpoints}, show that there are distinct walks with the same duration and number of jumps, having identical  endpoints. It is the nature of that approach to produce walks whose jump times are small perturbations of each other. In this section we will show by example that much more significant variation in jump times can still yield walks with jumps that have identical endpoints. We begin with Proposition \ref{approx one jump}, showing that a walk with a single, immediate jump can be well approximated for an arbitrary time by one in which a burst of two jumps occurs after a lag, followed by a regular sequence of jumps. We will then use a couple of tricks to turn a pair of such walks into a pair with identical durations and endpoints, in Proposition \ref{same duration jumps endvector}.

The family of examples produced here do not represent the full spectrum of possible ways in which walks with ``significantly different'' jump patterns may have identical endpoints; indeed, we do not feel that we have a firm handle on this issue at the moment. The construction we present makes non-generic choices at certain points for the sake of convenience. For instance, ``two jumps'' could easily be replaced by ``$n$ jumps'' in the paragraph above, for arbitrary $n\ge 2$. And we fix the jump angle at $\pi/2$ immediately below in order to take advantage of certain helpful hyperbolic trigonometric formulas for right-angled hyperbolic polygons.

In the lemma below, a \textit{pentagon} is a concatenation of distinct geodesic arcs $\alpha_1,\hdots,\alpha_5$---its \textit{sides}---such that $\alpha_i$ and $\alpha_{i+1}$ share an endpoint for each $i<5$, as do $\alpha_1$ and $\alpha_5$. Its \textit{vertices} are $v_i \doteq \alpha_i\cap\alpha_{i+1}$, for $i<5$, and $v_5 \doteq \alpha_5\cap\alpha_1$. Its \textit{angle} at $v_i$  ($i=1$ to $i=4$) is measured counterclockwise from $\alpha_{i+1}$ to $\alpha_i$, or for $i=5$, from $\alpha_1$ to $\alpha_5$. A pentagon is \textit{self-intersecting} if non-adjacent sides intersect.

\begin{lemma}\label{pentagon} For any $\ell>0$ there exist $0 < X_0 < X_1 \doteq \cosh^{-1}(\coth \ell\tanh(2\ell))$, 
such that for any $x\in (X_0,X_1)$ there is a self-intersecting pentagon $\alpha_1,\hdots,\alpha_5$ satisfying the following conditions:\begin{itemize}
    \item $\alpha_1$ has length $\ell$, $\alpha_2$ length $x$, and $\alpha_3$ length $2\ell$;
    \item the angles at $v_5$, $v_1$, $v_2$ and $v_3$ all equal $\pi/2$; and
    \item $\alpha_3$ intersects $\alpha_5$.
\end{itemize}
The angle $\delta$ at $v_4$ satisfies $\displaystyle{ \lim_{x\to X_0^+} \delta = 0}$, and $\displaystyle{\lim_{x\to X_1^-} \delta = \frac{\pi}{2} - \sin^{-1}\left(\frac{\cosh\ell}{\cosh(2\ell)}\right)}$.
\end{lemma}

\begin{proof} A quadrilateral in $\mathbb{H}^2$ with three right angles (called a ``Lambert quadrilateral'' in the classical geometry literature), is determined by the lengths of its two sides that have right angles at both of their vertices. If these two side lengths are $\ell$ and $x$ then the final vertex angle $\gamma < \pi/2$ satisfies:\begin{align}\label{gamma}
    \cos \gamma = \sinh \ell \sinh x.\quad \mbox{(See \cite[Thrm.~3.5.9]{Ratcliffe}.)} \end{align}
Note that this implies in particular that the product $\sinh\ell\sinh x$ can be at most 1. If this product is at least 1 then the two geodesics meeting the sides with lengths $\ell$ and $x$ at right angles to their endpoints do not meet in $\mathbb{H}^2$; if it equals $1$ then the two sides are ``parallel'', ie.~asymptotic, the quadrilateral has finite area, and we say it has a single ``ideal vertex''.

If $\sinh\ell \sinh x < 1$ then applying \cite[Thrm.~3.5.8]{Ratcliffe} to the Lambert quadrilateral above gives the relation between $\ell$, $\gamma$, and the length $y$ of the side opposite the one with length $\ell$ recorded in the first equality below. Equation (\ref{gamma}) then allows us to replace $\sin\gamma$ in the denominator, yielding the second equality:
\[ \cosh y = \frac{\cosh \ell}{\sin\gamma} = \frac{\cosh\ell}{\sqrt{1-\sinh^2 x\sinh^2\ell}}. \]
Note that if we fix $\ell$ and allow $x$ to vary, this formula defines $y$ as an \textit{increasing} function of $x$. Setting $y = 2\ell$ and solving for $x$  yields the formula for $X_1$ given in the Lemma's statement.

By the above, for any $x<X_1$, the following construction yields a geodesic arc $\alpha_3$ crossing a geodesic ray $\beta_5$ that contains $\alpha_5$: Let $\alpha_1$ and $\alpha_2$ be arcs of lengths $\ell$ and $x$, respectively, meeting at right angles at a point $v_1$; let $\alpha_3$ be a geodesic arc of length $2\ell$ meeting $\alpha_2$ at right angles at its endpoint $v_2$ opposite $v_1$, so that $\alpha_1$ and $\alpha_3$ belong to the same half-space bounded by the geodesic containing $\alpha_2$; and let $\beta_5$ be the geodesic ray meeting $\alpha_1$ at right angles at its endpoint $v_5$ opposite $v_1$, and contained in the same half-space bounded by the geodesic containing $\alpha_1$ as $\alpha_3$.

\begin{figure}[h]
 \begin{center}
    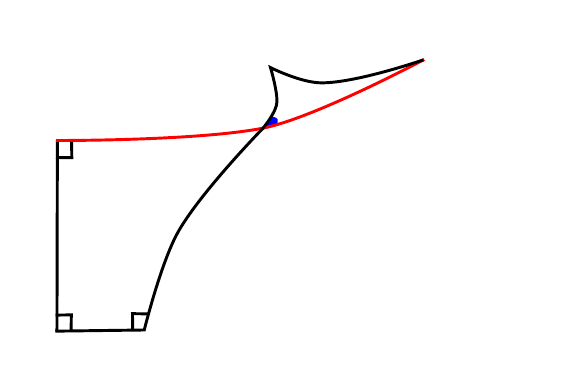
    \caption{Accompanies the proof of Lemma~\ref{pentagon}.}
\end{center}\label{F: mega pentagon}
\end{figure}

Still taking $x<X_1$, for $\alpha_3$ and $\beta_5$ as above, let $c = \alpha_3\cap \beta_5$ and let $\beta_3$ be the geodesic ray with endpoint $c$, contained in the geodesic containing $\alpha_3$, in the opposite half-space bounded by the geodesic containing $\beta_5$ from $\alpha_2$. There exists a point on $\beta_3$ at a distance $z>0$ from $c$ such that a geodesic ray through this point, at right angles to $\beta_3$ in the half-space that $\beta_5$ points into, is parallel to $\beta_5$. This point can be found because the geodesic through it forms a triangle, together with segments of $\beta_3$ and $\beta_5$, that has one ideal vertex and angles of $\pi/2$ and $\gamma$ (as defined in equation (\ref{gamma}) above) at its other two angles. The distance $z$ satisfies $\cosh z = 1/\sin \gamma$, see \cite[Thrm.~3.5.7]{Ratcliffe}. Again fixing $\ell$, and regarding $\gamma$ and hence $z = z(\gamma)$ as functions of $x$, we find that since $\gamma$ is a decreasing function of $x$, then $z$ is an increasing function of $x$.

As $x\to 0^+$, $\gamma\to \pi/2$, $y\to\ell$, and $z\to 0$; as $x\to X_1$, $y\to 2\ell$. Thus $y+z$ is less than $2\ell$ for $x$ near $0$ and greater than $2\ell$ for $x$ near $X_1$. Since $y$ and $z$ each increase with $x$ their sum does as well, so there is a unique $x=X_0$ with $0<X_0<X_1$ for which $y+z = 2\ell$, and for all $x\in(X_0,X_1)$, $y < 2\ell < y+z$. By the equations above and the angle addition law for hyperbolic cosine, $X_0$ satisfies:
\[ \cosh(2\ell) = \frac{\cosh\ell+\sinh\ell\sinh X_0\cosh X_0}{1-\sinh^2 X_0\sinh^2\ell}\]

Now given $x\in (X_0,X_1)$, let $\alpha_1$, $\alpha_2$, and $\alpha_3$ be geodesic arcs as described above, with lengths $\ell$, $x$, and $2\ell$, respectively; let $v_1 = \alpha_1\cap\alpha_2$ and $v_2 = \alpha_2\cap\alpha_3$; and let $\beta_5$ be the geodesic ray described above with its endpoint at the other endpoint $v_5$ of $\alpha_1$. Because $2\ell < y+z$, a geodesic ray from the other endpoint $v_3$ of $\alpha_3$, at right angles to $\alpha_3$ and pointing into the same half-space bounded by the geodesic containing it as $\beta_5$, does intersect $\beta_5$. Let $\alpha_4$ be the arc contained in this ray and joining $v_3$ to its point $v_4$ of intersection with $\beta_5$. Finally, let $\alpha_5$ be the arc of $\beta_5$ joining $v_4$ back to $v_5$. This completes the construction of the self-intersecting pentagon.
\end{proof}

\begin{proposition}\label{approx one jump}
    For any initial data $o$, $\bv$, $s$, $\theta$, $\ell$ of a walk with jumps, such that the jump angle $\theta = \pi/2$, there exist  $0<X_0<X_1$ (as defined above) such that for any $x\in(X_0,X_1)$ and $T_0>0$, there exist two walks with jumps sharing the given initial data and a common endpoint such that:\begin{itemize}
        \item the first has duration $T\ge T_0$ and a single jump at time $0$; and
        \item the second has duration $T_1 < T$, its first jump (a burst of two) at time $x/s$, and a total number of jumps that increases linearly with $T_0$.
    \end{itemize}
    More precisely: there exists $K>0$, depending only on $x$, such that $T\in [T_0,T_0+K)$, and $b = b(x) >0$ such that the second walk's jump times $(t_1,t_2,\hdots,t_n)$ satisfy $t_1 = x/s$, $t_2<K$, and $t_i = t_2+(i-2)b$ for each $i\geq 2$.
\end{proposition}

\begin{proof}
For the value of $\ell$ in the given initial data, let $0<X_0 < X_1$ be as in Lemma \ref{pentagon}. Then for $x\in (X_0,X_1)$, arrange the pentagon given by that result so that its vertex $v_1$ is at $o$ and $\bv$ points along the side $\alpha_2$ of length $x$ toward $v_2$. Then the side $\alpha_1$ of length $\ell$ is counterclockwise from $\alpha_2$ at $v_1 = o$. As in the proof of Lemma \ref{pentagon}, let $\beta_5$ be the ray sharing the other endpoint $v_0$ of $\alpha_1$ and containing the pentagon's side $\alpha_5$. Since the first walk with jumps has a single jump at time $0$, which has length $\ell$, its sole nontrivial walk segment is contained in $\beta_5$.

\begin{figure}[ht]
    \begin{center}
    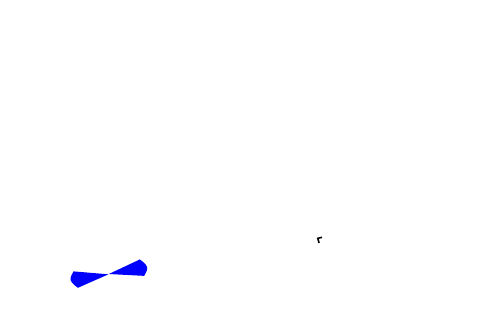
    \end{center}
    \caption{The triangle at the tip of the pentagon}
    \label{triangle}
\end{figure}

Since the second walk is prescribed to have a burst of two jumps at $t_1 = x/s$, its first walk segment is $\alpha_2$, its first jump segment is $\alpha_3$, and its second walk segment is contained in the geodesic ray $\beta_4$ that has an endpoint at $v_3$ and contains $\alpha_4$. We prescribe the second jump time $t_2$ to be larger than $\ell_4/s$, where $\ell_4$ is the length of $\alpha_4$---so the second walk segment contains $\alpha_4$---and with the property that the second jump axis is bisected by its intersection with $\beta_5$. Precisely: taking $d_2 = st_2 - \ell_4$, we must have\begin{align}\label{are too}
    \cosh d_2 = \frac{\sqrt{\sinh^2(\ell/2)+\sin^2\delta}}{\cosh(\ell/2)\sin \delta}, \end{align}
where $\delta$ is the pentagon's angle at its vertex $v_4$. This follows from the hyperbolic laws of sines and cosines; see Figure \ref{triangle}. There, the hyperbolic law of cosines gives that $\cosh h = \cosh d_2 \cosh(\ell/2)$, and the hyperbolic law of sines that $\sinh h = \sinh(\ell/2)/\sin\delta$. Substituting for $h$ in the first equation yields the formula (\ref{are too}).

If we now prescribe for each $i>2$ that $t_i - t_{i-1} = b \doteq 2d_2/s$, then the $i$th jump segment is also bisected by its intersection with $\beta_5$ for each such $i$. To see this, rotate the triangle of Figure \ref{triangle} by $180$ degrees about its vertex of intersection between the sides of length $h$ and $\ell/2$: the resulting isometric copy also has its side of length $h$ in $\beta_5$, its side of length $\ell/2$ is the other half of the second walk's second jump segment, and its side of length $d_2$ is the first half of the third walk segment. Rotating this triangle $180^{\circ}$ around its other non right-angled vertex yields the second half of the third walk segment and the first half of the third jump segment, and we proceed further by rotating around subsequent non right-angled vertices.

We now take $K = (\ell_5 + 2h)/s$, where $\ell_5$ is the length of the pentagon's side $\alpha_5$ and $h$ is as in Figure \ref{triangle}. Then for any $T_0>0$, there is a walk with given initial data, a single jump at time $0$, and duration in $[T_0, T_0+K)$ that terminates at one of the intersection points of $\beta_5$ with a walk segment described in the paragraph above. These are segments of the second walk with jumps, which we run for as long as it takes to meet the first.
\end{proof}

Note that the above construction of a specific pair of walks that have a common endpoint, has a biological interpretation:  one way in which the endpoint can fail to faithfully encode the jump times is that a jump at time zero can be mimicked by a sequence of jumps at later times.  

We now record an observation that will allow us to create ``rotationally symmetric'' walks with jumps.

\begin{lemma}\label{rotational symmetry}
For a walk with jumps with initial data $o$, $\bv$, $s$, $\theta$, $\ell$, and duration $T$, if the jump times $(t_1,\hdots,t_{2n})$ and burst vector $(j_1,\hdots,j_{2n})$ are \mbox{\rm symmetric} in the sense that $T/2 - t_i = t_{2n+1-i} - T/2$ and $j_i = j_{2n+1-i}$ for each $i\le n$, then:\begin{itemize}
    \item the associated walk with jumps path is itself symmetric under $\pi$-rotation around the walk's midpoint;
    \item denoting this rotation as $\rho$, the walk's endpoint and endvector are respectively $\rho(o)$ and $d\rho_o(-\bv)$; and
    \item the word $w$ supplied by Lemma \ref{wwj via isoms} that carries $(o,\bv)$ to the endpoint-endvector pair $(p,\bw)$ of the walk with jumps satisfies $\rho w\rho = w^{-1}$.
\end{itemize}\end{lemma}

\begin{proof}
We call the $i$th walk segment (as in Definition \ref{D: walkswithjumpspath}) $\alpha_i$, for $1\leq i\le 2n+1$, and the $j$th jump segment $\beta_j$ for $1\leq j\le 2n$. Note that $\alpha_i$ has length $\ell_i = s \delta_i$ for $1< i \leq 2n$, where $\delta_i = t_i - t_{i-1}$; and $\ell_1 = s t_1$ and $\ell_{2n+1} = s(T - t_{2n})$. By hypothesis: 
    \[ T/2-t_1 = t_{2n} - T/2,\quad\Rightarrow\quad T - t_{2n} = t_1,\quad\Rightarrow\quad \ell_1 = \ell_{2n+1}. \] 
Similarly, for $1<i\leq n$, $\ell_i = \ell_{2n+2-i}$ since:\begin{align}\label{delta eye}
    \delta_i & = \left[ (T/2 - t_i) - (T/2 - t_{i-1})\right] \\
        & = \left[ (t_{2n+1-i} - T/2) - (t_{2n+1 - (i-1)} - T/2)\right] = \left[ t_{2n+2-i} - t_{2n+1-i}\right] = \delta_{2n+2-i}.\nonumber\end{align}
The symmetry hypothesis on the jump vector similarly implies that $\beta_j$ and $\beta_{2n+1-j}$ have identical lengths for each $j$. 

It follows from this that the midpoint of the associated walk with jumps path is the midpoint of $\alpha_{n+1}$, so the rotation $\rho$ by an angle $\pi$ around this point takes $\alpha_{n+1}$ to itself. By construction of the walk with jumps, each of $\beta_n$ and $\beta_{n+1}$ intersects $\alpha_{n+1}$ at an angle of $\pi-\theta$ clockwise from the tangent vector pointing into $\alpha_{n+1}$ at their point of intersection. Since $\rho$ exchanges the endpoints of $\alpha_{n+1}$, and $\beta_n$ and $\beta_{n+1}$ have identical lengths, it also exchanges $\beta_n$ and $\beta_{n+1}$. Now $\alpha_n$ and $\alpha_{n+2}$ respectively intersect $\beta_n$ and $\beta_{n+1}$, each at an angle of $\pi-\theta$ \textit{counter}clockwise from the tangent vector pointing into $\beta_n$ or $\beta_{n+1}$, so $\rho$ exchanges $\alpha_n$ with $\alpha_{n+2}$ since $\ell_n  = \ell_{n+2}$.

Working our way outward from the midpoint as above, we ultimately find that $\rho$ takes the entire walks with jumps path to itself. The assertion about the walk's endpoint and endvector now follow from this and Remark \ref{R: endvector}.

We now interpret the symmetry condition in terms of the word $w$ supplied by Lemma \ref{wwj via isoms} that carries $(o,\bv)$ to the endpoint-endvector pair $(p,\bw)$ of the walk with jumps. For this walk, $w$ has the form
\[ w = g_{t_1}f^{j_1}g_{\delta_2}f^{j_2}\hdots g_{\delta_{2n}}f^{j_{2n}}g_{T-t_{2n}}, \]
where for each $i>1$, $\delta_i = t_i - t_{i-1}$ is the difference between successive jump times defined above. From (\ref{delta eye}) we have for each such $i$ that $\delta_i = \delta_{2n+2-i}$. In particular, $\delta_{n+2} = \delta_n$, $\delta_{n+3} = \delta_{n-1}$, and so forth through $\delta_{2n} = \delta_2$. Also using that $T-t_{2n}=t_1$, we obtain the following visibly symmetric form for $w$:
\[ w = \left(g_{t_1}f^{j_1}g_{\delta_2}\hdots g_{\delta_n}f^{j_n}\right)
    g_{\delta_{n+1}}
        \left(f^{j_n}g_{\delta_n}\hdots g_{\delta_2}f^{j_1}g_{t_1}\right) \]
We now write $w = w_0 w_1$ where $w_0 = g_{t_1}f^{j_1}g_{s_2}\hdots g_{s_n}f^{j_n}g_{s_{n+1}/2}$ is its ``first half'' and $w_1 = g_{s_{n+1}/2}f^{j_n}g_{s_n}\hdots g_{s_2}f^{j_1}g_{t_1}$ its second. By Lemma \ref{wwj via isoms}, $w_0$ carries the origin $o$ to the walk's midpoint. Therefore the rotation $\rho$ by $\pi$ about this point has the form $w_0 \rho_0 w_0^{-1}$, where $\rho_0$ is the $\pi$-rotation about $o$. Recalling that a $\pi$-rotation is its own inverse, we consider the conjugate of $w$ by $\rho$:\begin{align*}
    \rho w \rho^{-1} & = \left(w_0\rho_0w_0^{-1}\right) w_0w_1 \left(w_0\rho_0 w_0^{-1}\right) \\
    & = w_0 \rho_0 w_1 w_0 \rho_0 w_0^{-1} = w_0 (\rho_0 w_1 \rho_0)(\rho_0 w_0\rho_0) w_0^{-1} \end{align*}
We claim now that $\rho_0 w_i \rho_0 = w_{1-i}^{-1}$ for $i = 0$ and $1$, from which it will follow that $\rho w\rho = w^{-1}$ using the above. This uses the fact that $\rho_0$ conjugates $f$ to $f^{-1}$, and likewise $g_t$ for any $t\in\mathbb{R}$, since their axes run through its fixed point $o$. (One can check this fact directly by taking $o = \mathbf{0}$ and $\bv = \mathbf{e}_1$ in the Poincar\'e disk model and using the explicit descriptions below.)
\[ \rho_0 = \left(\begin{smallmatrix} i & 0 \\ 0 & -i \end{smallmatrix}\right)
    \quad\mbox{and}\quad
    g_t = \left(\begin{smallmatrix} \cosh(t/2) & \sinh(t/2) \\ \sinh(t/2) & \cosh(t/2) \end{smallmatrix}\right) \]
Now inserting $\rho_0\rho_0$, a form of the identity, between each syllable of $w_0$ for instance, we obtain:\begin{align*}
    \rho_0w_0\rho_0 & = \left(\rho_0g_{t_1}\rho_0\right)\left(\rho_0f^{j_1}\rho_0\right)\left(\rho_0g_{s_2}\rho_0\right)\hdots\left(\rho_0g_{s_{n+1}/2}\rho_0\right) \\
    & = g_{t_1}^{-1} f^{-j_1} g_{s_2}^{-1}\hdots g_{s_{n+1}/2}^{-1} = w_1^{-1}
    \end{align*}
Likewise for $w_1$, and the claim, and hence the lemma, holds.
\end{proof}

We now exploit Lemma \ref{rotational symmetry} to promote Proposition \ref{approx one jump}'s walks with jumps to a pair sharing an endpoint \textit{and endvector}.

\begin{proposition}\label{same endvector}
    For any initial data $o$, $\bv$, $s$, $\theta$, $\ell$ of a walk with jumps, such that the jump angle $\theta = \pi/2$, there exist $b,K>0$ and $0<X_0<X_1$ such that for any $T_0>0$ and $x\in(X_0,X_1)$, there exist two walks with jumps sharing the given initial data and a common endpoint and endvector:\begin{itemize}
        \item The first has duration $T_2\in [2T_0,2T_0+2K)$ and two jumps, at time $0$ and $t=T_2$, with burst vector $(1,1)$;
        \item The second has jump times $(t_1,\hdots,t_{2n})$ and burst vector $(2,1,\hdots,1,2)$ for some $n\ge 2$, where $t_1 = x/s$, $t_{2n} = 2T_1 - t_1$ and for $2\le i\leq n$, $t_i = t_2+(i-2)b$ and $t_{2n+1-i} = 2T_1 - t_i$, where $2T_1$ is its duration.
    \end{itemize}
\end{proposition}

\begin{proof} For the given initial data, let $0 < X_0 < X_1$ be supplied by Proposition \ref{approx one jump}. Then for $T_0>0$ and $x\in (X_0,X_1)$ let that result supply a duration $T$ for a first walk and jump times $(t_1,t_2,\hdots t_n)$ for a second, with burst vector $(2,1,\hdots,1)$ such that the two walks with jumps share an endpoint. Let $T_1$ be the duration of the second walk with jumps supplied by Proposition \ref{approx one jump}.

We construct the walks with jumps for the current result to be rotationally symmetric in the sense of Lemma \ref{rotational symmetry}, with those supplied by Proposition \ref{approx one jump} as initial segments (up to their midpoints). Thus we take the first to have duration $T_2 \doteq 2 T$ and single jumps at times $0$ and $T_2 = T_2-0$; and the second to have duration $2T_1$ and jump times and burst vector as given in the statement. This walk's jump data is thus symmetric, i.e., ~the jump times satisfy the equation $T_1 - t_i = t_{2n+1-i} - T_1$ for each $i$, and the burst vector entries also match, so by Lemma \ref{rotational symmetry} its associated walk-with-jumps path is symmetric by the rotation $\rho$ about its midpoint.

The first walk with jumps also has symmetric jump data, so its walk-with-jumps path is also rotationally symmetric. But these two paths have the same midpoint, since by construction their initial segments are the walks with jumps given by Proposition \ref{approx one jump} and these have a common endpoint. Hence the current walks with jumps have both a common endpoint $\rho(o)$ and endvector $d\rho_o(-\bv)$.
\end{proof}

The walks constructed in Proposition \ref{same endvector} have different durations and numbers of jumps, but this can be fixed by reversing their roles and doubling again.

\begin{proposition}\label{same duration jumps endvector}
    For any initial data $o$, $\bv$, $s$, $\theta$, $\ell$ of a walk with jumps, such that the jump angle $\theta = \pi/2$, there exist $b,K>0$ and $0<X_0<X_1$ such that for any $T_0>0$ and $x\in(X_0,X_1)$, there exist two walks with jumps having the given initial data and identical durations $T_2 + 2T_1$---for $T_2\in [2T_0,2T_0+2K)$ and some $T_1>0$---numbers of jumps, and endpoint and endvector:\begin{itemize}
        \item The first has jump times $(0,T_2,T_2+t_1,\hdots,T_2+t_{2n})$, for some $n\ge 2$, and burst vector $(1,1,2,1,\hdots,1,2)$;
        \item The second has jump times $(t_1,\hdots,t_{2n},2T_1,2T_1+T_2)$, for the same $n$, and burst vector $(2,1,\hdots,1,2,1,1)$. 
    \end{itemize}        
        In each case above, $t_1 = x/s$, $t_{2n} = 2T_1 - t_1$ and for $2\le i\leq n$, $t_i = t_2+(i-2)b$ and $t_{2n+1-i} = 2T_1 - t_i$.
\end{proposition}

\begin{proof}
    To construct the first and second walks here we use the corresponding walks from Proposition \ref{same endvector} as their initial segments, so that result supplies $T_2$, $T_1$, and the jump times $t_1,\hdots,t_n$. The initial segments thus have the same endpoint and endvector, after a duration of $T_2$ for the first and $2T_1$ for the second. For the terminal segment of the first walk here, we then use a copy of the \textit{second} walk supplied by Proposition \ref{same endvector}; and conversely, for the terminal segment of the second walk we use a copy of the previous Proposition's \textit{first} walk.
\end{proof}

%% file: selfintersectingpentagon.pdf_tex
\begingroup%
  \makeatletter%
  \providecommand\color[2][]{%
    \errmessage{(Inkscape) Color is used for the text in Inkscape, but the package 'color.sty' is not loaded}%
    \renewcommand\color[2][]{}%
  }%
  \providecommand\transparent[1]{%
    \errmessage{(Inkscape) Transparency is used (non-zero) for the text in Inkscape, but the package 'transparent.sty' is not loaded}%
    \renewcommand\transparent[1]{}%
  }%
  \providecommand\rotatebox[2]{#2}%
  \newcommand*\fsize{\dimexpr\f@size pt\relax}%
  \newcommand*\lineheight[1]{\fontsize{\fsize}{#1\fsize}\selectfont}%
  \ifx\svgwidth\undefined%
    \setlength{\unitlength}{275.13230211bp}%
    \ifx\svgscale\undefined%
      \relax%
    \else%
      \setlength{\unitlength}{\unitlength * \real{\svgscale}}%
    \fi%
  \else%
    \setlength{\unitlength}{\svgwidth}%
  \fi%
  \global\let\svgwidth\undefined%
  \global\let\svgscale\undefined%
  \makeatother%
  \begin{picture}(1,0.65852388)%
    \lineheight{1}%
    \setlength\tabcolsep{0pt}%
    \put(0,0){\includegraphics[width=\unitlength,page=1]{selfintersectingpentagon.pdf}}%
    \put(0.49247142,0.46289233){\color[rgb]{0,0,1}\makebox(0,0)[lt]{\lineheight{1.25000012}\smash{\begin{tabular}[t]{l}$\gamma$\end{tabular}}}}%
    \put(0,0){\includegraphics[width=\unitlength,page=2]{selfintersectingpentagon.pdf}}%
    \put(0.24282228,0.29723505){\color[rgb]{0,0.50196078,0}\makebox(0,0)[lt]{\lineheight{1.25000012}\smash{\begin{tabular}[t]{l}$y$\end{tabular}}}}%
    \put(0.3245126,0.2483738){\color[rgb]{0,0,0}\makebox(0,0)[lt]{\lineheight{1.25000012}\smash{\begin{tabular}[t]{l}$\alpha_3$\end{tabular}}}}%
    \put(0.54896949,0.52856338){\color[rgb]{0,0,0}\makebox(0,0)[lt]{\lineheight{1.25000012}\smash{\begin{tabular}[t]{l}$\alpha_4$\\\end{tabular}}}}%
    \put(0.37113921,0.50292048){\color[rgb]{0,0.50196078,0}\makebox(0,0)[lt]{\lineheight{1.25000012}\smash{\begin{tabular}[t]{l}$z$\end{tabular}}}}%
    \put(0,0){\includegraphics[width=\unitlength,page=3]{selfintersectingpentagon.pdf}}%
    \put(0.17504132,0.42942183){\color[rgb]{1,0,0}\makebox(0,0)[lt]{\lineheight{1.25000012}\smash{\begin{tabular}[t]{l}$\alpha_5$\end{tabular}}}}%
    \put(0.81314074,0.57194149){\color[rgb]{1,0,0}\makebox(0,0)[lt]{\lineheight{1.25000012}\smash{\begin{tabular}[t]{l}$\beta_5$\end{tabular}}}}%
    \put(0.46224014,0.60433236){\color[rgb]{0,0,0}\makebox(0,0)[lt]{\lineheight{1.25000012}\smash{\begin{tabular}[t]{l}$\beta_3$\end{tabular}}}}%
    \put(0.1537094,0.05276192){\color[rgb]{0,0,0}\makebox(0,0)[lt]{\lineheight{1.25000012}\smash{\begin{tabular}[t]{l}$\alpha_2$\end{tabular}}}}%
    \put(0.05735454,0.41322633){\color[rgb]{0,0,0}\makebox(0,0)[lt]{\lineheight{1.25000012}\smash{\begin{tabular}[t]{l}$v_5$\end{tabular}}}}%
    \put(0.25345919,0.04766364){\color[rgb]{0,0,0}\makebox(0,0)[lt]{\lineheight{1.25000012}\smash{\begin{tabular}[t]{l}$v_2$\end{tabular}}}}%
    \put(0.05355984,0.0554002){\color[rgb]{0,0,0}\makebox(0,0)[lt]{\lineheight{1.25000012}\smash{\begin{tabular}[t]{l}$v_1$\end{tabular}}}}%
    \put(0.05411549,0.22427988){\color[rgb]{0,0,0}\makebox(0,0)[lt]{\lineheight{1.25000012}\smash{\begin{tabular}[t]{l}$\alpha_1$\end{tabular}}}}%
    \put(0,0){\includegraphics[width=\unitlength,page=4]{selfintersectingpentagon.pdf}}%
    \put(0.59706741,0.26516984){\color[rgb]{0,0,0}\makebox(0,0)[lt]{\lineheight{1.25000012}\smash{\begin{tabular}[t]{l}$2\ell$\end{tabular}}}}%
    \put(0,0){\includegraphics[width=\unitlength,page=5]{selfintersectingpentagon.pdf}}%
    \put(-0.00247572,0.22487892){\color[rgb]{0,0.50196078,0}\makebox(0,0)[lt]{\lineheight{1.25}\smash{\begin{tabular}[t]{l}$\ell$\end{tabular}}}}%
    \put(0.15937757,0.00576604){\color[rgb]{0,0.50196078,0}\makebox(0,0)[lt]{\lineheight{1.25}\smash{\begin{tabular}[t]{l}$x$\end{tabular}}}}%
  \end{picture}%
\endgroup%

%% file: pentagonTip.pdf_tex
\begingroup%
  \makeatletter%
  \providecommand\color[2][]{%
    \errmessage{(Inkscape) Color is used for the text in Inkscape, but the package 'color.sty' is not loaded}%
    \renewcommand\color[2][]{}%
  }%
  \providecommand\transparent[1]{%
    \errmessage{(Inkscape) Transparency is used (non-zero) for the text in Inkscape, but the package 'transparent.sty' is not loaded}%
    \renewcommand\transparent[1]{}%
  }%
  \providecommand\rotatebox[2]{#2}%
  \newcommand*\fsize{\dimexpr\f@size pt\relax}%
  \newcommand*\lineheight[1]{\fontsize{\fsize}{#1\fsize}\selectfont}%
  \ifx\svgwidth\undefined%
    \setlength{\unitlength}{235.0216244bp}%
    \ifx\svgscale\undefined%
      \relax%
    \else%
      \setlength{\unitlength}{\unitlength * \real{\svgscale}}%
    \fi%
  \else%
    \setlength{\unitlength}{\svgwidth}%
  \fi%
  \global\let\svgwidth\undefined%
  \global\let\svgscale\undefined%
  \makeatother%
  \begin{picture}(1,0.67874902)%
    \lineheight{1}%
    \setlength\tabcolsep{0pt}%
    \put(0,0){\includegraphics[width=\unitlength,page=1]{pentagonTip.pdf}}%
    \put(0.3060199,0.12206698){\color[rgb]{0,0,1}\makebox(0,0)[lt]{\lineheight{1.25000012}\smash{\begin{tabular}[t]{l}$\delta$\end{tabular}}}}%
    \put(-0.00322028,0.0085597){\color[rgb]{1,0,0}\makebox(0,0)[lt]{\lineheight{1.25000012}\smash{\begin{tabular}[t]{l}$\alpha_5$\end{tabular}}}}%
    \put(0.00392984,0.14530472){\color[rgb]{0,0,0}\makebox(0,0)[lt]{\lineheight{1.25000012}\smash{\begin{tabular}[t]{l}$\alpha_4$\end{tabular}}}}%
    \put(0.37930783,0.06486658){\color[rgb]{0,0,0}\makebox(0,0)[lt]{\lineheight{1.25000012}\smash{\begin{tabular}[t]{l}$d_2$\end{tabular}}}}%
    \put(0.65369139,0.28204956){\color[rgb]{0,0,0}\makebox(0,0)[lt]{\lineheight{1.25000012}\smash{\begin{tabular}[t]{l}$\frac{\ell}2$\end{tabular}}}}%
    \put(0.51460916,0.45247363){\color[rgb]{0,0,0}\makebox(0,0)[lt]{\lineheight{1.25000012}\smash{\begin{tabular}[t]{l}$\frac{\ell}2$\end{tabular}}}}%
    \put(0.43650841,0.26596212){\color[rgb]{1,0,0}\makebox(0,0)[lt]{\lineheight{1.25000012}\smash{\begin{tabular}[t]{l}$h$\end{tabular}}}}%
    \put(0.8213752,0.47187093){\color[rgb]{1,0,0}\makebox(0,0)[lt]{\lineheight{1.25000012}\smash{\begin{tabular}[t]{l}$h$\end{tabular}}}}%
    \put(0,0){\includegraphics[width=\unitlength,page=2]{pentagonTip.pdf}}%
  \end{picture}%
\endgroup%

%% file: angleboundfigure.pdf_tex
\begingroup%
  \makeatletter%
  \providecommand\color[2][]{%
    \errmessage{(Inkscape) Color is used for the text in Inkscape, but the package 'color.sty' is not loaded}%
    \renewcommand\color[2][]{}%
  }%
  \providecommand\transparent[1]{%
    \errmessage{(Inkscape) Transparency is used (non-zero) for the text in Inkscape, but the package 'transparent.sty' is not loaded}%
    \renewcommand\transparent[1]{}%
  }%
  \providecommand\rotatebox[2]{#2}%
  \ifx\svgwidth\undefined%
    \setlength{\unitlength}{359.51171862bp}%
    \ifx\svgscale\undefined%
      \relax%
    \else%
      \setlength{\unitlength}{\unitlength * \real{\svgscale}}%
    \fi%
  \else%
    \setlength{\unitlength}{\svgwidth}%
  \fi%
  \global\let\svgwidth\undefined%
  \global\let\svgscale\undefined%
  \makeatother%
  \begin{picture}(1,0.57369476)%
    \put(0,0){\includegraphics[width=\unitlength,page=1]{angleboundfigure.pdf}}%
    \put(0.50067905,0.20861626){\color[rgb]{0,0,0}\makebox(0,0)[lb]{\smash{$p_i$}}}%
    \put(0.50067905,0.43809422){\color[rgb]{0,0,0}\makebox(0,0)[lb]{\smash{$q_i$}}}%
    \put(0,0){\includegraphics[width=\unitlength,page=2]{angleboundfigure.pdf}}%
    \put(0.71972615,0.19818541){\color[rgb]{1,0,0}\makebox(0,0)[lb]{\smash{$\gamma_i$}}}%
    \put(0.50067905,0.53197155){\color[rgb]{0,0,0}\makebox(0,0)[lb]{\smash{$\lambda_i$}}}%
    \put(0.94920411,0.45895581){\color[rgb]{0,0,0}\makebox(0,0)[lb]{\smash{$b$}}}%
    \put(0,0){\includegraphics[width=\unitlength,page=3]{angleboundfigure.pdf}}%
    \put(0.07599593,0.06631009){\color[rgb]{0,0,0}\makebox(0,0)[lb]{\smash{$a$}}}%
    \put(0,0){\includegraphics[width=\unitlength,page=4]{angleboundfigure.pdf}}%
    \put(0.36433346,0.0625849){\color[rgb]{0,1,0}\makebox(0,0)[lb]{\smash{$\theta$}}}%
    \put(0,0){\includegraphics[width=\unitlength,page=5]{angleboundfigure.pdf}}%
    \put(0.36507849,0.29206286){\color[rgb]{1,0.70196078,0}\transparent{0.99607843}\makebox(0,0)[lb]{\smash{$\pi-\theta$}}}%
    \put(0,0){\includegraphics[width=\unitlength,page=6]{angleboundfigure.pdf}}%
    \put(0.56028371,0.11771928){\color[rgb]{1,0.70196078,0}\transparent{0.99607843}\makebox(0,0)[lb]{\smash{$\pi-\theta$}}}%
  \end{picture}%
\endgroup%

%% file: SepEndpts.tex
\section{A criterion for having separate endpoints}\label{subsec: sepcrit} In this section we prove Theorem \ref{sep criterion} from the introduction. This will require the following trigonometric lemma:

\begin{lemma}\label{mabel} Let $(o,\bv,s,\theta=\pi/2,\ell)$ be the initial data of a walk with jumps having $n$ jumps and burst vector $(1,\hdots,1)$, and number the vertices of the associated walks-with-jumps path $p_i$ and $q_i$ as in Definition \ref{D: walkswithjumpspath}. For any $i<n$ such that $t_{i+1}-t_i \ge R$, the angle $\eta_i$ measured at the jump endpoint $q_i$, from the walk axis emanating from $q_i$ to the geodesic arc that joins $q_i$ to the next jump endpoint $q_{i+1}$ satisfies:
\[ \sin\eta_i \le \frac{\sinh\ell}{\sqrt{\cosh^2 (sR)\cosh^2\ell - 1}}. \]
\end{lemma}

\begin{figure}[ht]
\begin{center}
    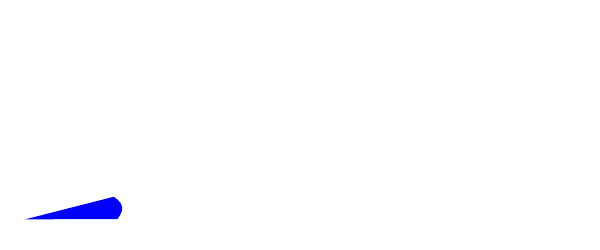
\end{center}
   \caption{A right triangle}
    \label{right tri}
\end{figure}

\begin{proof} The arcs in question emanating from $q_i$ form a right triangle, together with the $i+1$st jump axis, see Figure \ref{right tri}, whose base has length $X = s(t_{i+1}-t_i)$. The length $H$ of the triangle's hypotenuse satisfies $\cosh H = \cosh X\cosh\ell$, by the hyperbolic law of cosines, and by the hyperbolic law of sines the angle $\eta_i$ satisfies $\sin\eta_i = \sinh\ell/\sinh H$. Rewriting the first formula in terms of $\sinh H$ and substituting into the second yields the one given in the Lemma, but with equality and with $X$ replacing $sR$. Applying monotonicity of the resulting formula gives the result.
\end{proof}

We now recall the relevant definitions, also from the Introduction.

\begin{definition}
    \DistToResEps
\end{definition}

\begin{sep criterion thrm}\SepCriterion\end{sep criterion thrm}

We first prove a special case of Theorem \ref{sep criterion}, from which the general result will follow. It refers to the \emph{walk-with-jumps paths} from Definition \ref{D: walkswithjumpspath}.

\begin{lemma}\label{sep criterion case 0}
Let $R_{\min}>0$, the initial data $(o,\bv,s,\theta=\pi/2,\ell)$ of a walk with jumps, and a duration $T>0$ be given, and define $\epsilon$ as in Theorem \ref{sep criterion}. For a walk with jumps having the given initial data and duration, and minimum refractory length $R_{\min}$, which is distinct to resolution $\epsilon$ from a walk having its first jump at time $0$, the former walk's walk-with-jumps path does not intersect the latter walk's second walk axis (from Definition \ref{D: walk and jump axes}). In particular, the two walk-with-jumps paths intersect only at $o$.
\end{lemma}

\begin{figure}[ht]
\begin{center}
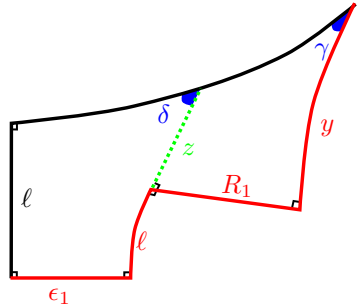
\caption{A non-convex hexagon.}
\label{dented hex}
\end{center}
\end{figure}

\begin{proof}
Consider two walks-with-jumps having minimum refractory length $R_{min}$, such that the first has a jump at time $0$ and the second has its first jump at time $\epsilon$ and its second at time $\epsilon+R_{min}$, the earliest possible. Then the initial segments of their walk-with-jumps paths belong to sides of the hexagon pictured in Figure \ref{dented hex}: those of the first belonging to the black sides, and those of the second, the red. Here $\epsilon_1 = s\epsilon$ and $R_1 = sR_{min}$. As in the Figure, we let $z$ be the length of the arc in the second walk's first jump axis joining the endpoint of its first jump segment to the first walk's walk axis. Hyperbolic trigonometry gives the following relations between the lengths of the sides of the resulting ``Lambert quadrilateral'' (one with three right angles) and the angle $\delta$:
\[ \tanh(\ell+z) = \cosh\epsilon_1 \tanh \ell,\quad\mbox{and}\quad \cos\delta = \sinh\ell\sinh\epsilon_1. \]
As in the proof of Proposition \ref{approx one jump} we note that if the product $\sinh\ell\sinh\epsilon_1$ were to exceed $1$, it would simply mean that the first walk's first walk axis and the second walk's first jump axis did not intersect. This would simultaneously imply that $\cosh\epsilon_1\tanh\ell > 1$, and hence that the left-hand equation above was also not satisfied.

However we aim to choose $\epsilon$ small enough that this does not occur; but large enough that the quantity $y$ of Figure \ref{dented hex} is at least $\ell$. To do this we apply angle-sum identities and simplify to solve the left-hand equation above for $z$:
\[ \tanh z = \frac{\cosh\epsilon_1-1}{\coth\ell - \tanh\ell\cosh\epsilon_1}.\]
We now observe that since the angle $\delta$ is less than $\pi/2$, there is a Lambert quadrilateral contained in the hexagon of Figure \ref{dented hex} that has its side of length $R_1$ as one side, a side intersecting that with length $z$, and another side properly contained in the one with length $y$. Applying the same trigonometric law to this quadrilateral, we obtain:
\[ \tanh y \ge \cosh R_1\tanh z = \frac{\cosh R_1(\cosh\epsilon_1-1)}{\coth\ell - \tanh\ell\cosh\epsilon_1}.\]
Thus to ensure that $y$ is at least $\ell$, it is enough to choose $\epsilon$ so that the right-hand side above is at least $\tanh\ell$. Setting it equal and solving for $\epsilon$ yields the value stated in the theorem.

We note that the right-hand lower bound for $y$ above is increasing in each of $\epsilon_1$ and $R_1$. This implies that if the second walk's first jump occurs at a time \textit{at least} the value of $\epsilon$ stated in the theorem, then its second jump will not cross the walk axis of the first walk with jumps, since this second jump occurs at a time at least $R_{min}$ after that of the first. We also note that the angle labeled $\gamma$ in Figure \ref{dented hex}, between the first walk's walk axis and the second walk's second jump axis, is less than the angle labeled $\delta$ in the figure. This is because the quadrilateral in the figure, with angles $\gamma$, $\pi-\delta$, and two right angles, has angle sum less than $2\pi$.

Let $q_i$ be the endpoint of the $i$th jump of the second walk, and let $\tau_i$ be the geodesic ray starting at $q_i$ and passing through $q_{i+1}$. 
Since the difference in jump times $t_{i+1}-t_i$ is at least  the minimum refractory length $R_{min}$, and $R_1 = sR_{min}$, Lemma \ref{mabel} asserts that the angle $\eta_i$ at $q_i$ made by $\tau_i$ and the walk axis emanating from $q_i$ satisfies:
\[ \sin\eta_i \le \frac{\sinh\ell}{\sqrt{\cosh^2 R_1\cosh^2\ell - 1}}. \]
We claim first that this bound implies that $\eta_i\le \pi/2-\delta$; or equivalently, comparing the bound with the formula for $\cos\delta$, that $\sinh\epsilon_1\sqrt{\cosh^2 R_1\cosh^2\ell - 1}\ge 1$. Substituting gives the left-hand side:
\begin{equation}\label{eq: trig expression bigger than 1} \sqrt{\left[\left(\frac{\cosh R_1+1}{\cosh R_1 - \tanh^2\ell}\right)^2 - 1\right]\left(\cosh^2 R_1\cosh^2\ell - 1\right)} \end{equation}
It is a straightforward computation to show that \eqref{eq: trig expression bigger than 1} exceeds $1$.

We now claim that as a consequence of this, the first and second walks with jumps do not cross.



Let $\rho$ be the walk axis of the first jump in Figure~\ref{dented hex} and let $\lambda_i$ be the $i$th jump axis of the second walk. We claim that if $\lambda_i$ intersects $\rho$, the intersection between $\lambda$ and $\rho$ in the quadrant bounded by $\lambda$ and $\rho$ containing $o$ is at an angle $\xi_i\le\delta$. 
Indeed, if $\lambda_i$ does not intersect $\rho$, then for all $j\ge i$, $\lambda_j$ does not intersect $\rho$ because $\lambda_i$ separates $\lambda_j$ from $\rho$, by Corollary~\ref{quadrants}.  
If $\lambda_{i-1}$ and $\lambda_i$ intersect $\rho$, then using the walk axis between $\lambda_{i-1},\lambda_i$, there is a quadrilateral whose angles are $\pi-\xi_{i-1},\,\frac\pi2,\,\frac\pi2$ and $\xi_i$ since the angle sum must be less than $2\pi$, we have $\xi_i\le \xi_{i-1}$. Since $\xi_1=\delta$, $\xi_i\le \delta$ by induction. 

\begin{figure}[ht]
    \centering
    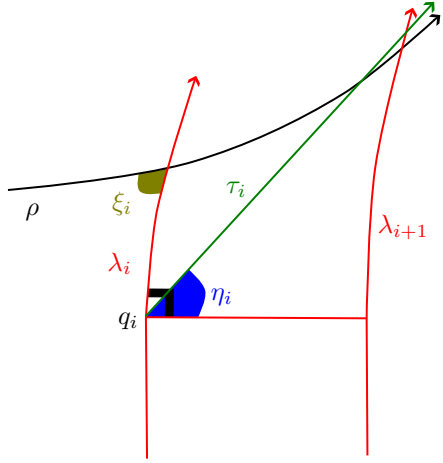
    \caption{The setup for showing that the first and second walk do not intersect. }
    \label{F: no future intersection}
\end{figure}

Aside from its first jump, the first walk travels entirely above $\rho$, see Corollary~\ref{quadrants} for details. Therefore, if the first and second walks with jumps cross, the second walk must cross $\rho$.  
Since $\eta_i$ is acute, if $q_i$ sits below $\rho$, the ray following the walk axis at $q_i$ heading in the direction of travel cannot cross $\rho$. Thus if the first and second walks with jumps cross, then there exist a pair $q_i,q_{i+1}$ that lie on opposite sides of $\rho$ since the second walk must intersect $\rho$ in a jump. Then $\tau_i$ crosses $\rho$. However we observe that if $\tau_i$ and $\rho$ cross, then there would be a triangle formed by $\rho$, $\tau_i$ and $\lambda_i$ where two of the interior angles are $\frac\pi2-\eta_i$, $\pi-\xi_i$. 
We have $\frac\pi2-\eta_i\ge \frac\pi2 - (\frac\pi2-\delta) =  \delta$, and $\pi-\xi_i\ge \pi-\delta$, so the interior angle sum of this triangle is at least $\pi-\delta+\delta = \pi$ so the third vertex of this triangle at the intersection of $tau_i$ and $\rho$ must be ideal. 
Therefore $\rho,\tau_i$ do not cross, so the first and second walks with jumps cannot intersect except at $o$.
\end{proof}

Theorem \ref{sep criterion} will now follow directly from the more precise result below.

\begin{theorem}\label{precise sep criterion}
Suppose $(s_1,\hdots,s_m)$ and $(t_1,\hdots,t_n)$ are the jump time sequences of two walks with jumps satisfying the hypotheses of Theorem \ref{sep criterion}. Define 
\[k_0 = \max\left(\{1\}\cup\{k\,|\,s_i = t_i\ \mbox{for all}\ i<k\}\right), \]
The intersection of the corresponding walk-with-jumps paths is the union of their $i^{\rm th}$ walk and jump segments, for all $i< k_0$, together with the shorter of the two $k_0^{\rm th}$ walk segments.
\end{theorem}

\begin{proof} If (say) $s_1 = 0$ and $t_1>0$ then $k_0$ defined as above equals $1$, the first walk-with-jumps path has a degenerate first walk segment consisting only of the origin $o$, and the Theorem's conclusion specializes to that of Lemma \ref{sep criterion case 0}. Our goal is to reduce to this case.

Let us assume now that either $m>n$ and $s_i = t_i$ for all $i\le n$, whence $k_0 = n+1$, or that $0<s_{k_0}<t_{k_0}$. Then for the two walk-with-jumps paths, constructed as in Definition \ref{D: walkswithjumpspath}, the points $p_i$ and $q_i$ coincide for all $i<k_0$. Moreover, the $k_0^{\rm th}$ walk segment of the walk with jump times $s_i$ is a subsegment of the other walk's $k_0^{\rm th}$ walk segment. Let $\gamma_{k_0}$ and $\lambda_{k_0}$ be the $k_0^{\rm th}$ walk and jump axes, respectively, of the walk with jump times $s_i$, defined as in \ref{D: walk and jump axes}. By Lemma \ref{wwj via isoms}, these are characterized by containing this walk's respective walk and jump \emph{segments}. We claim that for the half-planes $H_{k_0-1}^-$ and $K_{k_0}^{-}$ respectively bounded by these axes, identified in Corollary \ref{quadrants}, $H_{k_0-1}^-\cap K_{k_0}^{-}$ contains the union of the segments that were identified above as common to the two walk-with-jumps paths, and the intersection of complementary half-spaces $H_{k_0-1}^+\cap K_{k_0}^+$ contains the remainder of each of the two walk-with-jumps paths.

Corollary \ref{quadrants} asserts in particular that $H_{k_0-1}^-$ contains the duration-$s_{k_0-1}$ initial segment, and $K_{k_0-1}^-$ the duration-$s_{k_0}$ initial segment, of the walk-with-jumps path corresponding to the walk with jump times $s_i$. We note that $H_{k_0-1}^-$ also contains this walk's $k_0^{\rm th}$ walk segment, since this segment lies in the bounding geodesic $\gamma_{k_0}$. Therefore it contains all of the common segments identified above, and the duration-$k_0$ initial segment does as well.

The remainder of the walk-with-jumps path corresponding to the walk with jump times $s_i$, outside the union of its segments in common with the other walk-with-jumps path, consists of the union of its $k_0^{\rm th}$ jump segment with the complement (in the path) of the duration-$s_{k_0}$ initial segment. It is contained in the complement (in the path) of the duration-$s_{k_0-1}$ initial segment, so by Corollary \ref{quadrants}, it is contained in $H_{k_0-1}^+$. The Corollary moreover asserts that $K_{k_0}^+$ contains the remainder of the path beyond the duration-$s_{k_0}$ initial segment. Since the $k_0^{\rm th}$ jump segment is contained in the bounding geodesic $\lambda_{k_0}$ of $K_{k_0}^+$, it too contains the entire remainder of the path outside the union of its common segments with the other path.

Turning attention to the walk-with-jumps path corresponding to the walk with jump times $t_i$, we note that the half-planes $H_{k_0-1}^{\pm}$ defined for the first path play the same role for this path as for that one: $H_{k_0-1}^-$ contains the union of its duration-$t_{k_0-1}$ initial segment with its $k_0^{\rm th}$ walk segment, and $H_{k_0-1}^+$ contains the complement (in the path) of the duration-$t_{k_0-1}$ initial segment. It follows that $H_{k_0-1}^-$ contains the union of segments that the first path has in common with this one, and that $H_{k_0-1}^+$ contains the complement in this path of that union of segments. 

The half-planes $K_{k_0}^{\pm}$ defined for the first path do not play the same role for this path as for that one; however, their bounding geodesic $\lambda_{k_0}$ intersects this walk's $k_0^{\rm th}$ walk segment at an angle of $\theta$ in its interior, between its points of intersection with this walk's $k_0-1^{\rm st}$ and $k_0^{\rm th}$ jump axes---its endpoints. Arguing as in Lemma \ref{L: separation} we therefore find that $\gamma_{k_0}$ is disjoint from each of these jump axes, so by Corollary \ref{quadrants} $K_{k_0}^-$ contains the duration-$t_{k_0-1}$ initial segment and $K_{k_0}^+$ contains the complement (in the path) of the duration-$t_{k_0}$ initial segment. Moreover, $\gamma_{k_0}$ divides the first walk's $k_0^{\rm th}$ walk segment from its complement in the second walk's $k_0^{\rm th}$ walk segment. Therefore $K_{k_0}^+$ contains the complement in the second walk's walk-with-jumps path of the union of common segments. This establishes the claim's second half.

The claim implies that if the first walk's walk-with-jumps path has a point of intersection with that of the second outside the union of common segments, then this point of intersection must belong to the complement of the union of common segments in the second walk-with-jumps path as well. This is because these complementary sub-paths are walled off from the union of common segments within the quadrant $H_{k_0-1}^+\cap K_{k_0}^+$. We now observe that these complementary sub-paths are themselves walk-with-jumps paths of two related walks with jumps, to which we will apply Lemma \ref{sep criterion case 0} to finish the proof.

Let $p_{k_0}$ be as in Definition \ref{D: walkswithjumpspath} for the first walk's walk-with-jumps path, ie.~the terminal point of its $k_0^{\rm th}$ walk segment, and let $\mathbf{w}$ be the outward-pointing tangent vector to the $k_0^{\rm th}$ walk path at $p_{k_0}$. Then $p_{k_0}$ may be taken as the common origin, and $\mathbf{w}$ the common initial vector, for two walks with jumps that otherwise share the given walks' initial data, each have duration $T-s_{k_0}$, and whose jump times are $(s'_{k_0}=0,s'_{k_0+1},\hdots,s'_m)$ and $(t'_{k_0},\hdots,t'_{n})$, respectively, where $s'_j = s_j - s_{k_0}$ for $k_0\le j\le m$ and similarly for $t'_j$. The walk-with-jumps paths associated to these two walks are exactly the portions of the original walks' walk-with-jumps paths complementary to the union of their duration-$T_0$ initial segments, for all $T_0<s_{k_0}$. The two new walks satisfy all hypotheses of Lemma \ref{sep criterion case 0}, so that result's conclusion implies that their walk-with-jumps paths intersect only at the origin $p_{k_0}$. The current result follows.
\end{proof}

We finally return to the setting of Example \ref{E: binary tree}, using Theorem \ref{precise sep criterion} to embed a binary tree in $\mathbb{H}^2$ so that paths in the tree are taken to walk-with-jumps paths.

\begin{cor}\label{the ned}
For a given $R_{\min}>0$ and the initial data $(o,\bv,s,\theta=\pi/2,\ell)$ of a walk with jumps, define $\epsilon = \epsilon(R_{\min},s,\ell)$ as in Theorem \ref{sep criterion}, and let $m = \max\{R_{\min},\epsilon\}$. For a given $n\in\mathbb{N}$ and a binary tree $\mathcal{T}$ with root vertex $p_0$, there is an embedding of a finite subtree of $\mathcal{T}$ to $\mathbb{H}^2$, taking $p_0$ to $o$, with the property that the path in $\mathcal{T}$ encoded by any binary sequence $w = w_1\hdots w_n$ as in Example \ref{E: binary tree} (with each $w_i\in\{0,1\}$) maps to the walk-with-jumps path of a walk with jumps having the given initial data and duration $T\doteq n\cdot m$, with a jump at time $t$ if and only if $t=(i-1)m$ for some $i$ such that $w_i=1$.
\end{cor}

\begin{proof} We recall from Example \ref{E: binary tree} that $w=w_1\hdots w_n$ encodes a path in $\mathcal{T}$ as follows: at the root proceed left if $w_1=1$ and right if $w_1=0$. At the $(i-1)$th vertex of the path, proceed left if $w_i=1$ or right if $w_i=0$. We will map this path to the walk-with-jumps path of the walk with the given initial data and jump time sequence $((i_1-1)m,\hdots,(i_k-1)m)$, where $i_1 < \hdots < i_k$ is the set of $i\in\{1,\hdots,n\}$ such that $w_i = 1$. 

Note that such a walk has $k+1$ walk segments, where $k$ is the number of indices $i$ such that $w_i = 1$; and for $1< j\leq k$, the $j^{\rm th}$ walk segment has length $s(i_j-i_{j-1})m$. (And the first has length $s(i_1-1)m$, and the last, $s(n-i_k+1)m$.) The idea is to divide each walk segment into equal-length subsegments of length $sm$, then map the path encoded by $w$ via a continuous map taking $p_0$ to $o$, such that the $i^{\rm th}$ edge of the path in $\mathcal{T}$ maps to one of these subsegments if $w_i=0$ (ie.~$i\ne i_j$ for any $j\le k$) and to the union of a single jump segment with the subsequent walk sub-segment if $w_i=1$. We prescribe that each edge of the path in $T$ have length $1$ and require the restriction to each edge to have constant speed.

Let $\mathcal{T}_0$ be the union of the paths in $\mathcal{T}$ encoded by all binary sequences of length $n$, a finite subtree. There is a well-defined map on $\mathcal{T}_0$ that restricts on each such path to the one defined in the paragraph above: for any point $p\in\mathcal{T}_0$ there is a unique embedded edge path in $\mathcal{T}_0$ joining $p_0$ to $p$ (since $\mathcal{T}_0$ is a tree), and for any two paths encoded by binary sequences containing this as a sub-path, the maps to the corresponding walk-with-jumps paths agree on the sub-path, by Theorem \ref{precise sep criterion} and the canonical nature of the map's construction above.

Note that the walk with jumps corresponding to any path encoded by a binary sequence has minimum refractory length $R_{\min}$, since $m\ge R_{\min}$. And any two distinct such walks are distinct to resolution $\epsilon$, since $m\ge\epsilon$. It therefore follows again from Theorem \ref{precise sep criterion} that the map on $\mathcal{T}_0$ is an embedding.
\end{proof}

%% file: righttriangle.pdf_tex
\begingroup%
  \makeatletter%
  \providecommand\color[2][]{%
    \errmessage{(Inkscape) Color is used for the text in Inkscape, but the package 'color.sty' is not loaded}%
    \renewcommand\color[2][]{}%
  }%
  \providecommand\transparent[1]{%
    \errmessage{(Inkscape) Transparency is used (non-zero) for the text in Inkscape, but the package 'transparent.sty' is not loaded}%
    \renewcommand\transparent[1]{}%
  }%
  \providecommand\rotatebox[2]{#2}%
  \newcommand*\fsize{\dimexpr\f@size pt\relax}%
  \newcommand*\lineheight[1]{\fontsize{\fsize}{#1\fsize}\selectfont}%
  \ifx\svgwidth\undefined%
    \setlength{\unitlength}{287.61821362bp}%
    \ifx\svgscale\undefined%
      \relax%
    \else%
      \setlength{\unitlength}{\unitlength * \real{\svgscale}}%
    \fi%
  \else%
    \setlength{\unitlength}{\svgwidth}%
  \fi%
  \global\let\svgwidth\undefined%
  \global\let\svgscale\undefined%
  \makeatother%
  \begin{picture}(1,0.41606141)%
    \lineheight{1}%
    \setlength\tabcolsep{0pt}%
    \put(0,0){\includegraphics[width=\unitlength,page=1]{righttriangle.pdf}}%
    \put(0.21013028,0.06396668){\color[rgb]{0,0,1}\makebox(0,0)[lt]{\lineheight{1.25000012}\smash{\begin{tabular}[t]{l}$\eta_i$\end{tabular}}}}%
    \put(0.40636666,0.00612858){\color[rgb]{0,0,0}\makebox(0,0)[lt]{\lineheight{1.25000012}\smash{\begin{tabular}[t]{l}$X$\end{tabular}}}}%
    \put(0.76165804,0.19100411){\color[rgb]{0,0,0}\makebox(0,0)[lt]{\lineheight{1.25000012}\smash{\begin{tabular}[t]{l}$\ell$\end{tabular}}}}%
    \put(-0.00263139,0.04227735){\color[rgb]{0,0,0}\makebox(0,0)[lt]{\lineheight{1.25000012}\smash{\begin{tabular}[t]{l}$q_i$\end{tabular}}}}%
    \put(0.81020073,0.38930617){\color[rgb]{0,0,0}\makebox(0,0)[lt]{\lineheight{1.25000012}\smash{\begin{tabular}[t]{l}$q_{i+1}$\end{tabular}}}}%
    \put(0,0){\includegraphics[width=\unitlength,page=2]{righttriangle.pdf}}%
    \put(0.73790318,0.04021183){\color[rgb]{0,0,0}\makebox(0,0)[lt]{\lineheight{1.25000012}\smash{\begin{tabular}[t]{l}$p_{i+1}$\end{tabular}}}}%
    \put(0,0){\includegraphics[width=\unitlength,page=3]{righttriangle.pdf}}%
  \end{picture}%
\endgroup%

%% file: NonConvexHexagon.pdf_tex
\begingroup%
  \makeatletter%
  \providecommand\color[2][]{%
    \errmessage{(Inkscape) Color is used for the text in Inkscape, but the package 'color.sty' is not loaded}%
    \renewcommand\color[2][]{}%
  }%
  \providecommand\transparent[1]{%
    \errmessage{(Inkscape) Transparency is used (non-zero) for the text in Inkscape, but the package 'transparent.sty' is not loaded}%
    \renewcommand\transparent[1]{}%
  }%
  \providecommand\rotatebox[2]{#2}%
  \newcommand*\fsize{\dimexpr\f@size pt\relax}%
  \newcommand*\lineheight[1]{\fontsize{\fsize}{#1\fsize}\selectfont}%
  \ifx\svgwidth\undefined%
    \setlength{\unitlength}{169.21931026bp}%
    \ifx\svgscale\undefined%
      \relax%
    \else%
      \setlength{\unitlength}{\unitlength * \real{\svgscale}}%
    \fi%
  \else%
    \setlength{\unitlength}{\svgwidth}%
  \fi%
  \global\let\svgwidth\undefined%
  \global\let\svgscale\undefined%
  \makeatother%
  \begin{picture}(1,0.67495288)%
    \lineheight{1}%
    \setlength\tabcolsep{0pt}%
    \put(0,0){\includegraphics[width=\unitlength,page=1]{NonConvexHexagon.pdf}}%
    \put(0.32958794,0.4068681){\color[rgb]{0,0,1}\makebox(0,0)[lt]{\lineheight{1.25000012}\smash{\begin{tabular}[t]{l}$\delta$\end{tabular}}}}%
    \put(0,0){\includegraphics[width=\unitlength,page=2]{NonConvexHexagon.pdf}}%
    \put(0.67936465,0.55740857){\color[rgb]{0,0,1}\makebox(0,0)[lt]{\lineheight{1.25000012}\smash{\begin{tabular}[t]{l}$\gamma$\end{tabular}}}}%
    \put(0.47129062,0.25110026){\color[rgb]{1,0,0}\makebox(0,0)[lt]{\lineheight{1.25000012}\smash{\begin{tabular}[t]{l}$R_1$\end{tabular}}}}%
    \put(0,0){\includegraphics[width=\unitlength,page=3]{NonConvexHexagon.pdf}}%
    \put(0.08645568,0.01188821){\color[rgb]{1,0,0}\makebox(0,0)[lt]{\lineheight{1.25000012}\smash{\begin{tabular}[t]{l}$\epsilon_1$\end{tabular}}}}%
    \put(0.02589218,0.21903324){\color[rgb]{0,0,0}\makebox(0,0)[lt]{\lineheight{1.25000012}\smash{\begin{tabular}[t]{l}$\ell$\end{tabular}}}}%
    \put(0.28233838,0.12250781){\color[rgb]{1,0,0}\makebox(0,0)[lt]{\lineheight{1.25000012}\smash{\begin{tabular}[t]{l}$\ell$\end{tabular}}}}%
    \put(0.69296947,0.38755806){\color[rgb]{1,0,0}\makebox(0,0)[lt]{\lineheight{1.25000012}\smash{\begin{tabular}[t]{l}$y$\end{tabular}}}}%
    \put(0.38400722,0.33577177){\color[rgb]{0,1,0}\makebox(0,0)[lt]{\lineheight{1.25000012}\smash{\begin{tabular}[t]{l}$z$\end{tabular}}}}%
  \end{picture}%
\endgroup%

%% file: NoFutureIntersection.pdf_tex
\begingroup%
  \makeatletter%
  \providecommand\color[2][]{%
    \errmessage{(Inkscape) Color is used for the text in Inkscape, but the package 'color.sty' is not loaded}%
    \renewcommand\color[2][]{}%
  }%
  \providecommand\transparent[1]{%
    \errmessage{(Inkscape) Transparency is used (non-zero) for the text in Inkscape, but the package 'transparent.sty' is not loaded}%
    \renewcommand\transparent[1]{}%
  }%
  \providecommand\rotatebox[2]{#2}%
  \newcommand*\fsize{\dimexpr\f@size pt\relax}%
  \newcommand*\lineheight[1]{\fontsize{\fsize}{#1\fsize}\selectfont}%
  \ifx\svgwidth\undefined%
    \setlength{\unitlength}{229.20615203bp}%
    \ifx\svgscale\undefined%
      \relax%
    \else%
      \setlength{\unitlength}{\unitlength * \real{\svgscale}}%
    \fi%
  \else%
    \setlength{\unitlength}{\svgwidth}%
  \fi%
  \global\let\svgwidth\undefined%
  \global\let\svgscale\undefined%
  \makeatother%
  \begin{picture}(1,0.75168259)%
    \lineheight{1}%
    \setlength\tabcolsep{0pt}%
    \put(0,0){\includegraphics[width=\unitlength,page=1]{NoFutureIntersection.pdf}}%
    \put(0.33388095,0.26112435){\color[rgb]{0,0,1}\makebox(0,0)[lt]{\lineheight{1.25000012}\smash{\begin{tabular}[t]{l}$\eta_i$\end{tabular}}}}%
    \put(0,0){\includegraphics[width=\unitlength,page=2]{NoFutureIntersection.pdf}}%
    \put(0.17121437,0.40970993){\color[rgb]{0.50196078,0.50196078,0}\makebox(0,0)[lt]{\lineheight{1.25000012}\smash{\begin{tabular}[t]{l}$\xi_i$\end{tabular}}}}%
    \put(0,0){\includegraphics[width=\unitlength,page=3]{NoFutureIntersection.pdf}}%
    \put(0.02899366,0.40141108){\color[rgb]{0,0,0}\makebox(0,0)[lt]{\lineheight{1.25000012}\smash{\begin{tabular}[t]{l}$\rho$\end{tabular}}}}%
    \put(0.16390492,0.30961057){\color[rgb]{1,0,0}\makebox(0,0)[lt]{\lineheight{1.25000012}\smash{\begin{tabular}[t]{l}$\lambda_i$\end{tabular}}}}%
    \put(0.60802293,0.3714405){\color[rgb]{1,0,0}\makebox(0,0)[lt]{\lineheight{1.25000012}\smash{\begin{tabular}[t]{l}$\lambda_{i+1}$\end{tabular}}}}%
    \put(0.3589098,0.43715199){\color[rgb]{0,0.50196078,0}\makebox(0,0)[lt]{\lineheight{1.25000012}\smash{\begin{tabular}[t]{l}$\tau_i$\end{tabular}}}}%
    \put(0.18135078,0.22087366){\color[rgb]{0,0,0}\makebox(0,0)[lt]{\lineheight{1.25000012}\smash{\begin{tabular}[t]{l}$q_i$\end{tabular}}}}%
  \end{picture}%
\endgroup%

%% file: main.bbl
\begin{thebibliography}{10}

\bibitem{AronovVictor2004}
D.~Aronov and J.~D. Victor.
\newblock Non-euclidean properties of spike train metric spaces.
\newblock {\em Phys Rev E Stat Nonlin Soft Matter Phys}, 69(6 Pt 1):061905,
  2004.

\bibitem{BridsonHaefliger}
Martin~R. Bridson and Andr{\'e} Haefliger.
\newblock {\em Metric spaces of non-positive curvature}, volume 319 of {\em
  Grundlehren der Mathematischen Wissenschaften [Fundamental Principles of
  Mathematical Sciences]}.
\newblock Springer-Verlag, Berlin, 1999.

\bibitem{DerringtonKrauskopfLennie}
A.~M. Derrington, J.~Krauskopf, and P.~Lennie.
\newblock Chromatic mechanisms in lateral geniculate nucleus of macaque.
\newblock {\em J Physiol}, 357:241--65., 1984.

\bibitem{Fenchel}
Werner Fenchel.
\newblock {\em Elementary geometry in hyperbolic space}, volume~11 of {\em De
  Gruyter Studies in Mathematics}.
\newblock Walter de Gruyter \& Co., Berlin, 1989.
\newblock With an editorial by Heinz Bauer.

\bibitem{GoldShadlen}
J.~I. Gold and M.~N. Shadlen.
\newblock The neural basis of decision making.
\newblock {\em Annu Rev Neurosci}, 30:535--74, 2007.

\bibitem{HasinBrumshteinReview}
Y.~Hasin-Brumshtein, D.~Lancet, and T.~Olender.
\newblock Human olfaction: from genomic variation to phenotypic diversity.
\newblock {\em Trends Genet}, 25(4):178--84, 2009.

\bibitem{Hatcher}
Allen Hatcher.
\newblock {\em Algebraic topology}.
\newblock Cambridge University Press, Cambridge, 2002.

\bibitem{Hegde2008}
J.~Hegde.
\newblock Time course of visual perception: coarse-to-fine processing and
  beyond.
\newblock {\em Prog Neurobiol}, 84(4):405--39, 2008.

\bibitem{JacobsNirenbergRulingInAndOut}
A.~L. Jacobs, G.~Fridman, R.~M. Douglas, N.~M. Alam, P.~E. Latham, G.~T.
  Prusky, and S.~Nirenberg.
\newblock Ruling out and ruling in neural codes.
\newblock {\em Proc Natl Acad Sci U S A}, 106(14):5936--41, 2009.

\bibitem{LeeReview2008}
B.~B. Lee.
\newblock The evolution of concepts of color vision.
\newblock {\em Neurociencias}, 4(4):209--224, 2008.

\bibitem{LeeSmoothMflds}
John~M. Lee.
\newblock {\em Introduction to smooth manifolds}, volume 218 of {\em Graduate
  Texts in Mathematics}.
\newblock Springer, New York, second edition, 2013.

\bibitem{MochizukiShinomoto}
Y.~Mochizuki, T.~Onaga, H.~Shimazaki, T.~Shimokawa, Y.~Tsubo, R.~Kimura,
  A.~Saiki, Y.~Sakai, Y.~Isomura, S.~Fujisawa, K.~Shibata, D.~Hirai, T.~Furuta,
  T.~Kaneko, S.~Takahashi, T.~Nakazono, S.~Ishino, Y.~Sakurai, T.~Kitsukawa,
  J.~W. Lee, H.~Lee, M.~W. Jung, C.~Babul, P.~E. Maldonado, K.~Takahashi, F.~I.
  Arce-McShane, C.~F. Ross, B.~J. Sessle, N.~G. Hatsopoulos, T.~Brochier,
  A.~Riehle, P.~Chorley, S.~Grun, H.~Nishijo, S.~Ichihara-Takeda, S.~Funahashi,
  K.~Shima, H.~Mushiake, Y.~Yamane, H.~Tamura, I.~Fujita, N.~Inaba, K.~Kawano,
  S.~Kurkin, K.~Fukushima, K.~Kurata, M.~Taira, K.~Tsutsui, T.~Ogawa,
  H.~Komatsu, K.~Koida, K.~Toyama, B.~J. Richmond, and S.~Shinomoto.
\newblock Similarity in neuronal firing regimes across mammalian species.
\newblock {\em J Neurosci}, 36(21):5736--47, 2016.

\bibitem{OrsingherDeGregorioRandomFlights}
E.~Orsingher and A.~De~Gregorio.
\newblock Random flights in higher spaces.
\newblock {\em J. Theoret. Probab.}, 20(4):769--806, 2007.

\bibitem{Ratcliffe}
John~G. Ratcliffe.
\newblock {\em Foundations of hyperbolic manifolds}, volume 149 of {\em
  Graduate Texts in Mathematics}.
\newblock Springer, Cham, [2019] \copyright 2019.
\newblock Third edition [of 1299730].

\bibitem{ReichMechlerVictor}
D.~S. Reich, F.~Mechler, and J.~D. Victor.
\newblock Temporal coding of contrast in primary visual cortex: when, what, and
  why.
\newblock {\em J Neurophysiol}, 85(3):1039--50, 2001.

\bibitem{RichmondOptican1990}
B.~J. Richmond and L.~M. Optican.
\newblock Temporal encoding of two-dimensional patterns by single units in
  primate primary visual cortex. ii. information transmission.
\newblock {\em J Neurophysiol}, 64(2):370--80, 1990.

\bibitem{Sellers}
Peter~H. Sellers.
\newblock On the theory and computation of evolutionary distances.
\newblock {\em SIAM J. Appl. Math.}, 26:787--793, 1974.

\bibitem{SolomonLennie}
S.~G. Solomon and P.~Lennie.
\newblock The machinery of colour vision.
\newblock {\em Nat Rev Neurosci}, 8(4):276--86, 2007.

\bibitem{VictorPurpura1997}
J.D. Victor and K.P. Purpura.
\newblock Metric-space analysis of spike trains: theory, algorithms and
  application.
\newblock {\em Network}, 8:127--164, 1997.

\bibitem{WangXJ}
X.~J. Wang.
\newblock Decision making in recurrent neuronal circuits.
\newblock {\em Neuron}, 60(2):215--34, 2008.

\bibitem{ZaidiEtAl}
Q.~Zaidi, J.~Victor, J.~McDermott, M.~Geffen, S.~Bensmaia, and T.~A. Cleland.
\newblock Perceptual spaces: mathematical structures to neural mechanisms.
\newblock {\em J Neurosci}, 33(45):17597--602, 2013.

\bibitem{ZhouSmithSharpee}
Y.~Zhou, B.~H. Smith, and T.~O. Sharpee.
\newblock Hyperbolic geometry of the olfactory space.
\newblock {\em Sci Adv}, 4(8):eaaq1458, 2018.

\end{thebibliography}
